\documentclass[12pt]{amsart}
\usepackage{amsmath,amsthm,amssymb,latexsym,hyperref}
\usepackage[utf8]{inputenc}
\usepackage{amsfonts}
\usepackage{bbm}

\usepackage{xcolor}

\numberwithin{equation}{section}
\numberwithin{equation}{subsection}

\renewcommand*{\theequation}{%
  \ifnum\value{subsection}=0 %
    \thesection
  \else
    \thesubsection
  \fi
  .\arabic{equation}%
}

\newtheorem{thm}{Theorem}[section]
\newtheorem{lem}[thm]{Lemma}
\newtheorem{cor}[thm]{Corollary}

\title[Sign changes of the error term in the Piltz divisor problem]{Sign changes of the error term in the Piltz divisor problem}

\author{Siegfred Baluyot}
\address{American Institute of Mathematics\\
600 East Brokaw Road\\
San Jose, CA 95112, United States of America}
\email{\href{mailto:sbaluyot@aimath.org}{sbaluyot@aimath.org}}
\author{Cruz Castillo}
\address{ {Department of Mathematics \\
 University of Illinois at Urbana-Champaign \\
 1409 West Green Street, Urbana, IL 61801, United States of America} }
\email{\href{mailto:ccasti30@illinois.edu}{ccasti30@illinois.edu}}

\subjclass[2010]{11N37, 11M06. \\ \indent \textit{Keywords and phrases}: divisor function, Dirichlet divisor problem, moments, Riemann zeta-function, short intervals, fourth moment}

\begin{document}
\date{}

\begin{abstract}
We study the function $\Delta_k(x):=\sum_{n\leq x} d_k(n) - \mbox{Res}_{s=1} ( \zeta^k(s) x^s/s )$, where $k\geq 3$ is an integer, $d_k(n)$ is the $k$-fold divisor function, and $\zeta(s)$ is the Riemann zeta-function. For a large parameter $X$, we show that if the Lindel\"{o}f hypothesis is true, then there exist at least $X^{\frac{1}{k(k-1)}-\varepsilon}$ disjoint subintervals of $[X,2X]$, each of length $X^{1-\frac{1}{k}-\varepsilon}$, such that $|\Delta_k(x)|\gg x^{\frac{1}{2}-\frac{1}{2k}}$ for all $x$ in the subinterval. If the Riemann hypothesis is true, then we can improve the length of the subintervals to $\gg X^{1-\frac{1}{k}} (\log X)^{-k^2-2}$. These results may be viewed as higher-degree analogues of theorems of Heath-Brown and Tsang, who studied the case $k=2$, and Cao, Tanigawa, and Zhai, who studied the case $k=3$. The first main ingredient of our proofs is a bound for the second moment of $\Delta_k(x+h)-\Delta_k(x)$. We prove this bound using a method of Selberg and a general lemma due to Saffari and Vaughan. The second main ingredient is a bound for the fourth moment of $\Delta_k(x)$, which we obtain by combining a method of Tsang with a technique of Lester.
\end{abstract}

\maketitle

\section{Introduction and results}

For each integer $k\geq 2$, let $d_k(n)$ be the number of ways to write $n$ as a product $n_1n_2\cdots n_k$ with each $n_i$ a positive integer. Define
\begin{equation}\label{deltakdef}
\Delta_k (x):=\sum_{n\leq x} d_k (n) - \underset{s=1}{\mbox{Res}} \left(\frac{\zeta^k(s) x^s}{s}\right),
\end{equation}
where $\zeta(s)$ is the Riemann zeta-function. In 1955, Tong~\cite{tongI} showed for each $k\geq 2$ that there exists a constant $\beta_k>0$ such that, for all large enough $X$, $\Delta_k(x)$ changes sign at least once in the interval $[X,X+\beta_kX^{1-\frac{1}{k}}]$. The present article concerns the question: Can we shorten the length of this interval and still guarantee that $\Delta_k(x)$ changes sign at least once in the interval?

Heath-Brown and Tsang~\cite{hbtsang} have proven the existence of at least $\gg \sqrt{X}\log^5 X$ disjoint subintervals of $[X,2X]$, each of length a constant times $\sqrt{X} (\log X)^{-5}$, such that $|\Delta_2(x)|\gg x^{1/4}$ for all $x$ in any of the subintervals. Since $\Delta_2(x)$ is continuous except for jump discontinuities of size $d_2(n)\ll n^{\varepsilon}$, it follows that $\Delta_2(x)$ does not change sign in any of these subintervals. Thus, the case $k=2$ of Tong's theorem becomes false if we replace $\beta_2 \sqrt{X}$ by some constant times $\sqrt{X} (\log X)^{-5}$. In other words, the $k=2$ case of Tong's theorem is best possible up to factors of $\log X$.

In this paper, we prove under the assumption of the Riemann hypothesis (RH) that the the $k\geq 3$ case of Tong's theorem is best possible up to factors of $\log X$. For each integer $k\geq 2$, define the constant $C_k$ by
\begin{equation}\label{Ckdef}
C_k = \frac{1}{\pi} \Bigg(\frac{1}{2k} \sum_{n=1}^{\infty} \frac{d_k^2(n)}{n^{1+\frac{1}{k}}} \Bigg)^{1/2}.
\end{equation}
\begin{thm}\label{thm:deltaknosignchangeRH}
Assume the Riemann hypothesis and let $k\geq 3$ be an integer. Let $C_k$ be defined by \eqref{Ckdef}, and let $\varepsilon$ be an arbitrarily small positive constant. There exists constants $c_0,X_0>0$, with $c_0$ depending only on $k$ and $X_0$ depending only on $k$ and $\varepsilon$, such that if $X\geq X_0$, then there are at least $X^{\frac{1}{k(k-1)}-\varepsilon}$ disjoint subintervals of $[X,2X]$, each of length $c_0 \varepsilon X^{1-1/k}(\log X)^{-k^2-2}$, such that $|\Delta_k(x)|>(\frac{1}{2}C_k-\varepsilon)x^{\frac{1}{2}-\frac{1}{2k}}$ for all $x$ in the subinterval. In particular, $\Delta_k(x)$ does not change sign in any of these subintervals.
\end{thm}

If we assume the weaker Lindel\"{o}f hypothesis (LH) instead of RH, then we can prove that the $k\geq 3$ case of Tong's theorem is best possible up to a factor of $X^{\varepsilon}$.
\begin{thm}\label{thm:deltaknosignchange}
Assume the Lindel\"{o}f hypothesis and let $k\geq 3$ be an integer. Let $C_k$ be defined by \eqref{Ckdef}, and let $\xi$ and $\varepsilon$ be arbitrarily small positive constants. There exists a constant $X_0$ depending only on $k$, $\xi$, and $\varepsilon$ such that if $X\geq X_0$, then there are at least $ X^{\frac{1}{k(k-1)}+\xi-\varepsilon}$ disjoint subintervals of $[X,2X]$, each of length $X^{1-\frac{1}{k}-\xi}$, such that $|\Delta_k(x)|>(\frac{1}{2}C_k-\varepsilon)x^{\frac{1}{2}-\frac{1}{2k}}$ for all $x$ in the subinterval. In particular, $\Delta_k(x)$ does not change sign in any of these subintervals.
\end{thm}

Note that, similarly to the result of Heath-Brown and Tsang, Theorems~\ref{thm:deltaknosignchangeRH} and \ref{thm:deltaknosignchange} do not rule out the possibility that some of the disjoint subintervals may have a union that is contained in a longer subinterval on which $\Delta_k(x)$ does not change sign. On the other hand, Tong's theorem implies that this longer subinterval cannot have length larger than $\beta_k (2X)^{1-\frac{1}{k}}$.

To prove Theorems~\ref{thm:deltaknosignchangeRH} and \ref{thm:deltaknosignchange}, we will use the method of Heath-Brown and Tsang~\cite{hbtsang} for detecting intervals on which $\Delta_k(x)$ does not change sign. Their method requires bounds for the fourth moment of $\Delta_k(x)$ and the second moment of $\Delta_k(x+h)-\Delta_k(x)$. We provide such bounds by proving Theorems~\ref{thm:dkinshortintervalsnosup}, \ref{thm:dkinshortintervalsRH}, and \ref{thm:deltakfourthmmnt} below. A lot of research has been put towards understanding these moments and other properties of $\Delta_k(x)$ in recent decades.

Historically, a great deal of work has been done towards finding upper bounds for the order of magnitude of $\Delta_k(x)$. The well-known Dirichlet divisor problem concerns finding the value of inf$\{\theta : \Delta_2(x) \ll x^{\theta} \text{ for all } x\geq 1\}$. More generally, the \textit{Piltz divisor problem} asks for the value of the real number $\alpha_k$ defined by $\alpha_k:=$inf$\{\theta : \Delta_k(x) \ll x^{\theta} \text{ for all } x\geq 1\}$. The current record for the smallest upper bound for $\alpha_2$ is $\alpha_2 \leq 131/416$, due to Huxley~\cite{huxley1,huxley2}. Kolesnik~\cite{kolesnik} has shown that $\alpha_3\leq 43/96$, and upper bounds for $\alpha_k$ for $k\geq 4$ have been obtained by Ivi\'{c}~\cite[Theorems~13.2~and~13.3]{ivic}. Ford~\cite[p.~567]{Ford} has improved these bounds for large $k$. The Lindel\"{o}f hypothesis is equivalent to the statement that $\alpha_k \leq 1/2$ for all $k\geq 2$~\cite[Theorem~13.4]{titchmarsh}. It is known that $\alpha_k \geq (k-1)/(2k)$ \cite[Theorem~12.6(B)]{titchmarsh}, and Titchmarsh~\cite[\S 12.4]{titchmarsh} conjectures that $\alpha_k=(k-1)/(2k)$. Thus, our results show for $k\geq 3$ that $|\Delta_k(x)|$ reaches its conjectured upper bound within a factor of $x^{\varepsilon}$ for all $x$ inside many subintervals of $[X,2X]$ of length $\gg X^{1-1/k}(\log X)^{-k^2-2}$ (under RH) or $\gg X^{1-\frac{1}{k}-\varepsilon}$ (under LH). The current best omega result is due to Soundararajan~\cite{sound}, who has shown, by refining ideas of Hafner~\cite{hafner}, that
\begin{equation*}
\Delta_k(x)=\Omega\Big((x\log x)^{\frac{k-1}{2k}}(\log\log x)^{\frac{k+1}{2k}(k^{2k/(k+1)} -1)}(\log\log\log x)^{-\frac{1}{2}-\frac{k-1}{4k}} \Big).
\end{equation*}

Though the order of magnitude of $\sup\{|\Delta_k(x)| : x\in [1,X]\}$ is not known, the average size of $|\Delta_k(x)|$ is more well-understood. Cram\'{e}r~\cite{cramer} has proved an asymptotic formula for the second moment of $\Delta_2$, while Tong~\cite{tong} has shown that, unconditionally for $k=3$ and assuming the Lindel\"{o}f hypothesis for $k\geq 4$,
\begin{equation}\label{Tong2ndmoment}
\int_X^{2X} \Big(\Delta_k(x) \Big)^2 \,dx \sim \int_X^{2X} \big( C_k x^{\frac{1}{2}-\frac{1}{2k}}\big)^2 \,dx
\end{equation}
as $X\rightarrow\infty$, where $C_k$ is defined by \eqref{Ckdef}. The error term in Cram\'{e}r's asymptotic formula has been examined more closely by Lau and Tsang~\cite{lautsang1,lautsang2,tsang2ndmoment} and Ge and Gonek~\cite{fan}. Tsang~\cite{tsang4thmoment} has proved asymptotic formulas for the third and fourth moments of $\Delta_2$. Zhai ~\cite{zhai1,zhai2} improved the bounds for the error terms in Tsang's asymptotic formulas, and also proved asymptotic formulas for the $m$th moments of $\Delta_2$ for $5 \leq m\leq 9$. Furthermore, Ivi\'{c}~\cite[Chapter~13]{ivic} has obtained bounds for higher moments of $\Delta_2$ and $\Delta_3$. We shall prove a conditional upper bound for the fourth moment of $\Delta_k(x)$ for all $k\geq 3$ (Theorem~\ref{thm:deltakfourthmmnt} below) and use it as one of the main ingredients in our proofs of Theorems~\ref{thm:deltaknosignchangeRH} and \ref{thm:deltaknosignchange}. For further interesting research on moments and various other properties of $\Delta_2$, see the informative survey~\cite{tsangsurvey}.

While moments of $\Delta_k$ have been extensively studied, much work has also been done towards understanding the mean square of $\Delta_k(x;h):=\Delta_k(x+h)-\Delta_k(x)$ with $h$ a parameter. Moments of $\Delta_k$ present data about the size of $\Delta_k(x)$, while moments of $\Delta_k(x;h)$ give information about the fluctuations of $\Delta_k$. Jutila~\cite{jutila} has proved that
\begin{equation*}
\frac{1}{X}\int_X^{2X} \Big( \Delta_2(x+h)-\Delta_2(x)\Big)^2 \,dx \asymp h \log^3 \left( \frac{\sqrt{X}}{h}\right)
\end{equation*}
for $X^{\varepsilon}\leq h \leq X^{\frac{1}{2}-\varepsilon}$, while Ivi\'{c}~\cite{ivic3} improved this result by proving an asymptotic formula when $1\ll h \leq \frac{1}{2}\sqrt{X}$. For $k\geq 3$, Ivi\'{c}~\cite{ivic2} has proved bounds for the mean square of $\Delta_k(x;h)$ that depend on an arbitrary real number $\delta\geq 0$ satisfying
\begin{equation}\label{deltamomentbound}
\int_0^{\tau} |\zeta(\tfrac{1}{2}+\delta+it)|^{2k} \,dt \ll_{\varepsilon} {\tau}^{1+\varepsilon} \text{ as }\tau\rightarrow \infty,\text{ for all fixed }\varepsilon>0
\end{equation}
(where we allow the implied constant to depend on $\varepsilon$). His theorem states that if $k\geq 3$ is a fixed integer and \eqref{deltamomentbound} holds for $\delta=0$, then
\begin{equation}\label{eqn:ivicbound1}
\frac{1}{X}\int_X^{2X} \Big( \Delta_k(x+h)-\Delta_k(x)\Big)^2 \,dx \ll_{k ,\varepsilon} h^{4/3} X^{ \varepsilon}
\end{equation}
for $X^{\varepsilon}\leq h \leq X^{1-\varepsilon}$, while if $\delta>0$ satisfies \eqref{deltamomentbound} and $\eta>0$ is a constant, then
\begin{equation}\label{eqn:ivicbound2}
\frac{1}{X}\int_X^{2X} \Big( \Delta_k(x+h)-\Delta_k(x)\Big)^2 \,dx \ll_{k ,\eta,\varepsilon} h^{2} X^{-\frac{1}{3}\eta+\varepsilon}
\end{equation}
for $X^{2\delta+\eta} \leq h\leq X^{1-\varepsilon}$. More recently, Cao, Tanigawa, and Zhai~\cite{CTZ} have proved that if \eqref{deltamomentbound} holds for  $\delta =0$, then 
\begin{equation}\label{eqn:CTZbound1}
  \frac{1}{X}\int_{X}^{2X}(\Delta_k(x+h)-\Delta_k(x))^2dx \ll
  \begin{cases}
  hX^{\varepsilon} & \text{if } X^{1-\frac{1}{k}-\varepsilon} \ll h \ll X \\
  X^{\varepsilon}(h + X^{1-\frac{3}{k}}) & \text{if } 1\ll h \ll X^{1-\frac{1}{k}-\varepsilon}.
\end{cases}
\end{equation}
They also prove for $k=3$ that (unconditionally)
\begin{equation}\label{eqn:CTZbound2}
 \frac{1}{X} \int_{X}^{2X}(\Delta_3(x+h)-\Delta_3(x))^2dx \ll 
\begin{cases}
  X^{\varepsilon}(h + X^{1/3}h^{1/3}+X^{5/9}) & \text{if } X^{4/9} \ll h \leq X \\
  X^{1/3+\varepsilon}h^{1/2} & \text{if } 1\ll h \ll X^{4/9}.
\end{cases}
\end{equation}

If $h$ is instead equal to $x/T$ with $T$ a parameter such that $2\leq T\leq X$, then an argument implicit in Milinovich and Turnage-Butterbaugh~\cite{micah} leads to
\begin{equation*}
\frac{1}{X}\int_X^{2X} \bigg(\Delta_k \bigg(x+\frac{x}{T}\bigg) -\Delta_k(x) \bigg)^2\,dx \ll \frac{X}{T}(\log T)^{k^2}
\end{equation*}
via a method of Selberg~\cite{selberg} under the assumption of RH (see also \cite[(1.2)]{lester}). This is close to the true order of magnitude, as Lester~\cite{lester} has shown for certain constants $b_k$ that
\begin{equation*}
\begin{split}
\frac{1}{X} \int_X^{2X} \bigg(\Delta_k \bigg(x+\frac{x^{1-\frac{1}{k}}}{L}\bigg) -\Delta_k(x) \bigg)^2\,dx = \frac{b_k X^{1-\frac{1}{k}}}{L} (\log L)^{k^2-1} + O\bigg(\frac{ X^{1-\frac{1}{k}}}{L} (\log L)^{k^2-2}\bigg) 
\end{split}
\end{equation*}
unconditionally for $k=3$ and $2\leq L \ll X^{\frac{1}{12}-\varepsilon}$, and assuming LH for $k\geq 3$ and $2\leq L \ll X^{\frac{1}{k(k-1)}-\varepsilon}$. This agrees with a conjecture of Keating, Rodgers, Roditty-Gershon, and Rudnick~\cite{krrr}, who studied the analogous problem in function fields and used their results to predict for each integer $k\geq 3$ that if $h=X^{\vartheta}$ with $\vartheta$ a fixed real number in $(0,1-1/k)$, then
\begin{equation}\label{eqn:krrr}
\frac{1}{X}\int_X^{2X} \Big( \Delta_k(x+h)-\Delta_k(x)\Big)^2 \,dx \sim a_k\mathcal{P}_k(\vartheta) H(\log X)^{k^2-1}
\end{equation}
as $X\rightarrow \infty$, where $a_k$ is a constant depending only on $k$ and $\mathcal{P}_k$ is a specific piecewise polynomial of degree $k^2-1$. Through their conjecture, Keating et al. have found an interesting connection between the mean square of $\Delta_k(x;h)$ and averages of coefficients of characteristic polynomials of random matrices. Bettin and Conrey~\cite{bettinconrey} have shown that the conjecture \eqref{eqn:krrr} of Keating et al. would follow from a (yet unproved) conjecture for moments of $\zeta(s)$.

We refine the argument of Milinovich and Turnage-Butterbaugh~\cite{micah} and combine the method of Selberg~\cite{selberg} with a lemma due to Saffari and Vaughan~\cite{saffarivaughan} to bound the mean square of $\Delta_k(x;h)$ with the parameter $h$ independent of the variable $x$. Our results improve Ivi\'{c}'s~\cite{ivic2} bounds \eqref{eqn:ivicbound1} and \eqref{eqn:ivicbound2} for all $h$, and also improve Cao, Tanigawa, and Zhai's bounds \eqref{eqn:CTZbound1} and \eqref{eqn:CTZbound2} for small enough $h$. We will apply our bounds to our proofs of Theorems~\ref{thm:deltaknosignchangeRH} and \ref{thm:deltaknosignchange}.

\begin{thm}\label{thm:dkinshortintervalsnosup}
Let $k\geq 3$ be an integer, and let $\delta\geq 0$ be a real number satisfying \eqref{deltamomentbound}. Suppose further that $\varepsilon$ is an arbitrarily small positive constant. If $\,1\leq h\leq X/8$, then
$$
\frac{1}{X}\int_X^{2X} \Big( \Delta_k(x+h)-\Delta_k(x)\Big)^2 \,dx \ll h X^{2\delta+\varepsilon},
$$
with implied constant depending only on the implied constant in \eqref{deltamomentbound}.
\end{thm}

A theorem of Heath-Brown~\cite{heathbrown} (see also \cite[\S 7.22]{titchmarsh}) implies that $\delta=1/12$ satisfies \eqref{deltamomentbound} with $k=3$ and $\delta=1/8$ satisfies \eqref{deltamomentbound} with $k=4$. Various $\delta$ satisfying \eqref{deltamomentbound} for other $k$ may be deduced from Theorem~8.4 of Ivi\'{c}~\cite{ivic}, and Ford~\cite[p.~567]{Ford} has found smaller $\delta$ than these for large $k$. Using these values for $\delta$ in the application of Theorem~\ref{thm:dkinshortintervalsnosup} in Section~\ref{section:main} leads to an unconditional proof of the existence of a subinterval of $[X,2X]$ with length $X^{1-\frac{1}{k}-2\delta-\varepsilon}$ such that $|\Delta_k(x)|>(\frac{1}{2}C_k-\varepsilon)x^{\frac{1}{2}-\frac{1}{2k}}$ for all $x$ in the subinterval. However, finding a nontrivial lower bound for the number of such subintervals using the methods in Section~\ref{section:main} requires a strong upper bound for the fourth moment of $\Delta_k(x)$. The unconditional existence of many such subintervals for $k=3$ has been recently proved by Cao, Tanigawa and Zhai \cite{CTZ} (see the paragraph containing \eqref{HBbound} below for details). A well-known fact is that LH is equivalent to the statement that $\delta=0$ satisfies \eqref{deltamomentbound} for all $k$~\cite[Theorem~13.2]{titchmarsh}. From this and the aforementioned theorem of Heath-Brown~\cite{heathbrown} for $k=3$, we deduce the following two corollaries of Theorem~\ref{thm:dkinshortintervalsnosup}.

\begin{cor}\label{cor:d3inshortintervals}
Suppose that $\varepsilon$ is an arbitrarily small positive constant. If $\,1\leq h\leq X/8$, then (unconditionally)
$$
\frac{1}{X}\int_X^{2X} \Big( \Delta_3(x+h)-\Delta_3(x)\Big)^2 \,dx \ll_{\varepsilon} h X^{\frac{1}{6}+\varepsilon}.
$$
\end{cor}

\begin{cor}\label{cor:dkinshortintervalsLH}
Assume the Lindel\"{o}f hypothesis. Let $k\geq 3$ be a fixed integer, and suppose that $\varepsilon$ is an arbitrarily small positive constant. If $\,1\leq h\leq X/8$, then
$$
\frac{1}{X}\int_X^{2X} \Big( \Delta_k(x+h)-\Delta_k(x)\Big)^2 \,dx \ll_{k ,\varepsilon} h X^{ \varepsilon}.
$$
\end{cor}

Corollary~\ref{cor:d3inshortintervals} improves \eqref{eqn:CTZbound2} for $h \ll X^{1/3}$, while Corollary~\ref{cor:dkinshortintervalsLH} recovers \eqref{eqn:CTZbound1} for $k=3$ and improves \eqref{eqn:CTZbound1} for $k\geq 4$ and $h \ll X^{1-\frac{3}{k}}$.

By refining a method of Soundararajan~\cite{sound2}, Harper~\cite{harper} has proved that the Riemann hypothesis implies
\begin{equation*}
\int_0^{\tau} |\zeta(\tfrac{1}{2}+it)|^{2k} \,dt \ll_k \tau (\log \tau)^{k^2} \ \text{ as }\tau\rightarrow \infty
\end{equation*}
for all positive integers $k$. We may use this in place of \eqref{deltamomentbound} in our proof of Theorem~\ref{thm:dkinshortintervalsnosup} and arrive at the following theorem. We will use this to prove Theorem~\ref{thm:deltaknosignchangeRH}.

\begin{thm}\label{thm:dkinshortintervalsRH}
Assume the Riemann hypothesis. If $k\geq 3$ is a fixed integer and $\,1\leq h\leq X/8$, then
$$
\frac{1}{X}\int_X^{2X} \Big( \Delta_k(x+h)-\Delta_k(x)\Big)^2 \,dx \ll_{k}  h \log^{k^2}\left( \frac{X}{h}\right).
$$
\end{thm}

By the conjecture \eqref{eqn:krrr} of Keating et al., we expect that
$$
\frac{1}{X}\int_X^{2X} \Big( \Delta_k(x+h)-\Delta_k(x)\Big)^2 \,dx \ll_{k}  h (\log X)^{k^2-1}.
$$
If we assume this and LH, then we can deduce the conclusion of Theorem~\ref{thm:deltaknosignchangeRH} with the length of the subintervals improved to $c_0\varepsilon X^{1-1/k} (\log X)^{-k^2-1}$.

More than giving intervals on which $\Delta_k(x)$ does not change sign, Theorems~\ref{thm:deltaknosignchangeRH} and \ref{thm:deltaknosignchange} provide a lower bound for the measure of the set of all $x\in[X,2X]$ for which $|\Delta_k(x)|>(\frac{1}{2}C_k-\varepsilon)x^{\frac{1}{2}-\frac{1}{2k}}$. Heath-Brown and Tsang~\cite{hbtsang} do this for $k=2$ and show that $|\Delta_2(x)|>(\frac{1}{2}C_2-\varepsilon)x^{1/4}$ on a subset of $[X,2X]$ whose measure is $\gg X$. To deduce this lower bound for the measure, Heath-Brown and Tsang use an estimate for the fourth moment of $\Delta_2$ due to Tsang~\cite{tsang4thmoment}, who applied the Erd\"{o}s-Tur\'{a}n inequality and van der Corput's bound for exponential sums to prove the asymptotic formula
\begin{equation*}
\frac{1}{X}\int_2^X \big( \Delta_2(x)\big)^4\,dx = \frac{3}{64\pi^4} \!\!\! \sum_{\substack{1\leq n,m,k,\ell <\infty \\ \sqrt{n}+\sqrt{m} = \sqrt{k}+\sqrt{\ell}}} \!\!\!\!\!\! \frac{d_2(n)d_2(m) d_2(k)d_2(\ell) }{ (nmk\ell)^{3/4}} \,X + O\Big( X^{\frac{22}{23}+\varepsilon}\Big).
\end{equation*}

We combine Tsang's technique with the method of Lester~\cite{lester} to find a conditional bound for the fourth moment of $\Delta_k$. We shall apply this bound in our proofs of Theorems~\ref{thm:deltaknosignchangeRH} and \ref{thm:deltaknosignchange} to deduce a lower bound for the number of disjoint subintervals on which $|\Delta_k(x)|>(\frac{1}{2}C_k-\varepsilon)x^{\frac{1}{2}-\frac{1}{2k}}$.

\begin{thm}\label{thm:deltakfourthmmnt}
Assume the Lindel\"{o}f hypothesis, and let $\varepsilon>0$ be an arbitrarily small positive constant. If $k\geq 3$ and $X\geq 1$, then
\begin{equation*}
\frac{1}{X}\int_X^{2X} \big( \Delta_k(x)\big)^4\,dx \ll_{k,\varepsilon} X^{2-\frac{1}{k-1}+\varepsilon}.
\end{equation*}
\end{thm}
Our proof of the case $k=3$ of Theorem~\ref{thm:deltakfourthmmnt} can in fact be made unconditional (see the remark below Lemma~\ref{lemma:Ik4thmoment} in Section~\ref{section:lester}). However, Ivi\'{c}~\cite[Theorem~13.10]{ivic} has proved through a different method that
\begin{equation}\label{ivicbound}
\frac{1}{X}\int_X^{2X} \big( \Delta_3(x)\big)^4\,dx \ll_{\varepsilon} X^{\frac{139}{96}+\varepsilon}
\end{equation}
by applying Kolesnik's~\cite{kolesnik} pointwise bound $\Delta_3(x)\ll x^{\frac{43}{96}+\varepsilon}$. This bound for the fourth moment of $\Delta_3$ is stronger than the case $k=3$ of Theorem~\ref{thm:deltakfourthmmnt}. The current best unconditional bound for large $k$ is due to Ivi\'{c} and Zhai~\cite{IZ}, who proved for $k\geq 4$ that
\begin{align*}
\frac{1}{X}\int_X^{2X} \big( \Delta_k(x)\big)^4\,dx \ll_{\epsilon} X^{2-\frac{2}{k}+\varepsilon} + X^{4 -\frac{16}{2k+1} +\varepsilon }.
\end{align*}
The conjecture
\begin{equation*}
\Delta_k(x) \ll x^{\frac{1}{2}-\frac{1}{2k}+\varepsilon}
\end{equation*}
of Titchmarsh~\cite[\S 12.4]{titchmarsh}, if true, would imply that
\begin{equation*}
\frac{1}{X}\int_X^{2X} \big( \Delta_k(x)\big)^4\,dx \ll_{k,\varepsilon} X^{2-\frac{2}{k}+\varepsilon}.
\end{equation*}
If we assume this and LH (resp. RH), then we can deduce the conclusion of Theorem~\ref{thm:deltaknosignchange} (resp. \ref{thm:deltaknosignchangeRH}) with the lower bound for the number of disjoint subintervals improved to $X^{\frac{1}{k}+\xi-\varepsilon}$ (resp. $X^{\frac{1}{k}-\varepsilon}$).

Using Corollary~\ref{cor:d3inshortintervals} and \eqref{ivicbound} in place of Corollary~\ref{cor:dkinshortintervalsLH} and Theorem~\ref{thm:deltakfourthmmnt}, respectively, in our arguments in Section~\ref{section:main} for $k=3$, we are able to prove unconditionally the existence of $\gg X^{\frac{37}{96}-\varepsilon}$ disjoint subintervals of $[X,2X]$, each of length $\gg X^{\frac{1}{2}-\varepsilon}$, such that $|\Delta_3(x)|>(\frac{1}{2}C_3-\varepsilon)x^{1/3}$ for all $x$ in the subinterval. However, Cao, Tanigawa, and Zhai~\cite{CTZ} have proven the stronger result that there are $\gg X^{\frac{1}{2}-\varepsilon}$ such subintervals. They also prove under the assumption of the Lindel\"{o}f hypothesis that there are $\gg X^{\frac{1}{3}-\varepsilon}$ disjoint subintervals of $[X,2X]$, each of length $\gg X^{\frac{2}{3}-\varepsilon}$, such that $|\Delta_3(x)|>(\frac{1}{2}C_3-\varepsilon)x^{1/3}$ for all $x$ in the subinterval. This result is stronger than the case $k=3$ of Theorem~\ref{thm:deltaknosignchange}, which implies the existence of only $\gg X^{\frac{1}{6}-\varepsilon}$ such subintervals. They are able to obtain these stronger results for $k=3$ by showing that $|\Delta_3(x)|>(\frac{1}{2}C_3-\varepsilon)x^{1/3}$ on a subset of $[X,2X]$ whose measure is $\gg X^{1-\varepsilon}$. They do so by applying the bound
\begin{equation}\label{HBbound}
\frac{1}{X}\int_X^{2X} |\Delta_3(x)|^3 \,dx \ll X^{1+\varepsilon}
\end{equation}
due to Heath-Brown~\cite{heathbrown2}. Using this idea, we may improve the $k=3$ case of Theorem~\ref{thm:deltaknosignchangeRH} and deduce the following.
\begin{thm}\label{thm:delta3nosignchangeRH}
Assume the Riemann hypothesis. Let $C_3$ be defined by \eqref{Ckdef} with $k=3$, and let $\varepsilon$ be an arbitrarily small positive constant. There exists an absolute constant $c_0>0$ and a constant $X_0>0$ depending only on $\varepsilon$ such that if $X\geq X_0$, then there are at least $X^{\frac{1}{3}-\varepsilon}$ disjoint subintervals of $[X,2X]$, each of length $c_0 \varepsilon X^{2/3}(\log X)^{-11}$, such that $|\Delta_3(x)|>(\frac{1}{2}C_3-\varepsilon)x^{1/3}$ for all $x$ in the subinterval. In particular, $\Delta_3(x)$ does not change sign in any of these subintervals.
\end{thm}

The rest of the paper is organized as follows. In Section~\ref{sec:notation}, we set some notations and conventions that hold throughout this work. In Section~\ref{sec:lemmas}, we prove some technical lemmas that are used in the proofs of our main results. We use Lester's method in Section~\ref{section:lester} to bound moments involving the contribution of large frequencies in the trigonometric polynomial approximation to $\Delta_k(x)$. We prove Theorems~\ref{thm:dkinshortintervalsnosup} and \ref{thm:dkinshortintervalsRH} in Section~\ref{section:meansquare}. We prove Theorem~\ref{thm:deltakfourthmmnt} in Section~\ref{section:fourthmoment}, and prove Theorems~\ref{thm:deltaknosignchangeRH}, \ref{thm:deltaknosignchange}, and \ref{thm:delta3nosignchangeRH} in Section~\ref{section:main}.

\section*{Acknowledgements}

The authors would like to thank Wenguang Zhai for several helpful comments on a previous version of this article, and for informing us of the paper \cite{CTZ}. The first author was partially supported by the National Science Foundation grant DMS-1854398. The second author was partially supported by the Graduate College Master’s Fellowship program at the University of Illinois, the Alfred P. Sloan Foundation’s MPHD Program, awarded in 2021, and by the National Science Foundation Graduate Research Fellowship Program under Grant No. DGE 21-46756.

\section{Notations and conventions}\label{sec:notation}

For the rest of this paper, $k$ denotes an integer $\geq 3$. Most of our arguments will work for $k=2$, but this special case is already well-understood in the context of our main results through the works of Heath-Brown and Tsang~\cite{hbtsang}, Ivi\'{c}~\cite{ivic3}, and Tsang~\cite{tsang4thmoment}.

We follow standard convention in analytic number theory and use $\varepsilon$ to denote an arbitrarily small positive constant whose value may vary from one line to the next. We allow implied constants to depend on $\varepsilon$ and $k$ without necessarily indicating so. We will sometimes display the dependence of implied constants on $\varepsilon$, $k$, or other quantities by using subscripts such as those in $A \ll_B C$ or $r=O_s(t)$. Implied constants will never depend on the parameters $H,T,X,Y$.

We use $e(x)$ to denote $e^{2\pi i x}$. For $x,V,Y,T>0$, we define $Q_k(x;V)$ and $I_k(x;Y,T)$ by
\begin{equation}\label{Qkdef}
Q_k(x;V) := \frac{x^{\frac{1}{2}-\frac{1}{2k}}}{\pi \sqrt{k}}\sum_{n\leq V/x} \frac{d_k(n)}{n^{\frac{1}{2}+\frac{1}{2k}}} \cos\bigg(2\pi k (nx)^{1/k}+ \frac{(k-3)\pi}{4}\bigg)
\end{equation}
and
\begin{equation}\label{Ikdef}
I_k(x;Y,T):=\text{\upshape{Re}} \Bigg\{ \frac{1}{\pi i } \int_{\frac{1}{2} +iY}^{\frac{1}{2} +iT}\zeta^k(s) \frac{x^s}{s}\,ds \Bigg\}.
\end{equation}

\section{Lemmata}\label{sec:lemmas}

The first of two key ingredients in our proofs of Theorems~\ref{thm:dkinshortintervalsnosup} and \ref{thm:dkinshortintervalsRH} is a method of Selberg~\cite{selberg} that uses the Plancherel theorem to express a weighted mean square of $\Delta_k(x+x/T)-\Delta(x)$ in terms of a weighted $2k$th moment of $\zeta(s)$ (see equation \eqref{eqn:selberg} below). In carrying out Selberg's method, we use the following lemma.

\begin{lem}\label{lemma:perron}
Let $\Delta_k(x)$ be defined by \eqref{deltakdef}, and define $\Delta_k^*(x)$ by
\begin{equation}\label{deltak*def}
\Delta_k^*(x) =\frac{\Delta_k(x+)+\Delta_k(x-)}{2}.
\end{equation}
If $\delta$ satisfies $0\leq \delta<1/2$ and \eqref{deltamomentbound}, then there exists a sequence $T_1,T_2,\dots$ of positive real numbers such that $T_{m}\in [2^m,2^{m+1}]$ for each $m$ and
\begin{equation*}
\Delta_k^*(x)  = \lim_{m\rightarrow \infty} \frac{1}{2\pi i} \int_{\frac{1}{2}+\delta-iT_m}^{\frac{1}{2}+\delta+iT_m} \frac{x^s}{s} \zeta^k(s) \,ds
\end{equation*}
for all $x>0$.
\end{lem}
\begin{proof}
Let $g(y)=0$ for $0<y<1$, $g(y)=1/2$ for $y=1$, and $g(y)=1$ for $y>1$. Then Perron's formula (see, for example, the lemma in \S 17 of Davenport~\cite{davenport}) and the definitions \eqref{deltakdef} and \eqref{deltak*def} imply
\begin{equation}\label{eqn: applyperron}
\begin{split}
& \Delta_k^*(x) + \underset{s=1}{\mbox{Res}} \left(\frac{\zeta^k(s) x^s}{s}\right) = \sum_{n=1}^{\infty} d_k(n) g\left( \frac{x}{n}\right) \\
& = \frac{1}{2\pi i} \int_{2-iT}^{2+iT} \frac{x^s}{s} \zeta^k(s) \,ds +O\left( \frac{d_k(x)}{T}\right) + O\Bigg( x^2 \sum_{\substack{n=1 \\ n\neq x}}^{\infty} \frac{d_k(n)}{n^2} \min \left\{ 1, \frac{1}{T|\log (x/n)|}\right\}\Bigg)
\end{split}
\end{equation}
for any $x,T>0$, where we define $d_k(x)=0$ if $x$ is not a positive integer. We move the line of integration and use the residue theorem to write
\begin{equation}\label{eqn: residue1}
\frac{1}{2\pi i} \int_{2-iT}^{2+iT} \frac{x^s}{s} \zeta^k(s) \,ds = \underset{s=1}{\mbox{Res}} \left(\frac{\zeta^k(s) x^s}{s}\right) + \frac{1}{2\pi i} \Bigg(  \int_{\frac{1}{2}+\delta-iT}^{\frac{1}{2}+\delta+iT} + \int_{2-iT}^{\frac{1}{2}+\delta-iT} + \int_{\frac{1}{2}+\delta+iT}^{2+iT} \Bigg)\frac{x^s}{s} \zeta^k(s) \,ds.
\end{equation}
To estimate the latter two integrals, which are along horizontal line segments, we define
\begin{equation}\label{fkdeltadef}
f_{k,\delta}(T) := \Bigg( \int_{2-iT}^{\frac{1}{2}+\delta-iT} + \int_{\frac{1}{2}+\delta+iT}^{2+iT} \Bigg) \left| \frac{\zeta^k(s)}{s}  \,ds \right|.
\end{equation}
If $m$ is a positive integer, then \eqref{fkdeltadef} and the Cauchy-Schwarz inequality imply
\begin{align}
\int_{2^m}^{2^{m+1}} f_{k,\delta}(T) \,dT
& \ll \frac{1}{2^m} \int_{\frac{1}{2}+\delta}^2 \int_{2^m}^{2^{m+1}} |\zeta(\sigma+iT)|^k \,dT\,d\sigma \notag \\
& \ll \frac{1}{2^{m/2}} \int_{\frac{1}{2}+\delta}^2 \Bigg( \int_{2^m}^{2^{m+1}} |\zeta(\sigma+iT)|^{2k} \,dT\Bigg)^{1/2} \,d\sigma. \label{eqn: fkdeltabound}
\end{align}
By \eqref{deltamomentbound} and convexity (see, for example, \S 7.8 of Titchmarsh~\cite{titchmarsh}), it holds that
\begin{equation*}
\int_{2^m}^{2^{m+1}} |\zeta(\sigma+iT)|^{2k} \,dT \ll 2^{m(1+\varepsilon)}
\end{equation*}
uniformly for all positive integers $m$ and all $\sigma$ in the interval $[\frac{1}{2}+\delta,2]$. From this and \eqref{eqn: fkdeltabound}, we deduce that
\begin{equation*}
\int_{2^m}^{2^{m+1}} f_{k,\delta}(T) \,dT \ll 2^{m\varepsilon}
\end{equation*}
uniformly for all positive integers $m$. Since $f_{k,\delta}(T)$ is nonnegative by \eqref{fkdeltadef}, it follows that for each positive integer $m$ there is a $T_m$ in the interval $[2^m,2^{m+1}]$ such that
\begin{equation*}
f_{k,\delta}(T_m) \ll 2^{m(-1+\varepsilon)}.
\end{equation*}
From this, the definition \eqref{fkdeltadef} of $f_{k,\delta}$, and the triangle inequality, we arrive at
\begin{equation*}
\Bigg(  \int_{2-iT_m}^{\frac{1}{2}+\delta-iT_m} + \int_{\frac{1}{2}+\delta+iT_m}^{2+iT_m} \Bigg)\frac{x^s}{s} \zeta^k(s) \,ds \ll 2^{m(-1+\varepsilon)} \max\{ x^{\frac{1}{2}+\delta},x^2\}
\end{equation*}
for all $m$. The lemma now follows from this, \eqref{eqn: applyperron}, and \eqref{eqn: residue1}.
\end{proof}

While the first of two key ingredients in our proofs of Theorems~\ref{thm:dkinshortintervalsnosup}  and \ref{thm:dkinshortintervalsRH} is Selberg's method, the second key ingredient is the following lemma, which allows us to bound the mean square of $\Delta_k(x+h)-\Delta_k(x)$ in terms of the mean square of $\Delta_k(x+x/T)-\Delta_k(x)$. This lemma is essentially due to Saffari and Vaughan~\cite{saffarivaughan}, and we use a version due to Goldston and Suriajaya~\cite{goldstonsuriajaya} (see also \cite{goldstonvaughan}).

\begin{lem}[Goldston and Suriajaya~\cite{goldstonsuriajaya}, Lemma~3]\label{lem:saffarivaughan}
If $f:\mathbb{R}\rightarrow \mathbb{C}$ is integrable, $X>0$, and $0< h\leq X/4$, then
\begin{equation*}
\int_{X/2}^X |f(t+h)-f(t)|^2\,dt \leq \frac{2X}{h}\int_0^{8h/X}\int_0^X |f(t+\beta t)-f(t)|^2\,dt\,d\beta.
\end{equation*}
\end{lem}
\begin{proof}
See the proof of Lemma~3 in \cite{goldstonsuriajaya}. The said proof also applies to the case when $0<h\leq 1$.
\end{proof}

The following lemma is a slight modification of Lemma~2.5 of \cite{lester}, and is the starting point of our proof of Theorem~\ref{thm:deltakfourthmmnt}.

\begin{lem}\label{lemma:lester2.5}
Assume the Lindel\"{o}f hypothesis. Let $\Delta_k(x)$ be defined by \eqref{deltakdef}. If $x,T\geq 1$ and $1\leq Y\leq \min\{x,T\}$, then
\begin{equation*}
\Delta_k(x) =Q_k( x; Y^k/(2\pi)^k ) + I_k(x;Y,T) + E_k(x;Y,T),
\end{equation*}
where $Q_k$ is defined by \eqref{Qkdef}, $I_k$ is defined by \eqref{Ikdef}, and
\begin{equation}\label{Ekbound}
\begin{split}
E_k(x;Y,T)
\ll  x^{1+\varepsilon} Y^{-\frac{k}{2}-\frac{1}{2}} + x^{\varepsilon}Y^{\frac{k}{2}-1} + x^{\frac{1}{2} } Y^{ -1+\varepsilon} + x^{1+\varepsilon} T^{-1+\varepsilon} .
\end{split}
\end{equation}
\end{lem}
\begin{proof}
The proof is similar to that of \cite[Lemma~2.5]{lester}, but we provide it since our situation is slightly different. A standard argument using Perron's formula leads to
\begin{equation*}
\sum_{n\leq x} d_k(n) = \frac{1}{2\pi i}\int_{1+\varepsilon -iT}^{1+\varepsilon +iT} \zeta^k(s) \frac{x^s}{s}\,ds + O \bigg( x^{\varepsilon}+\frac{x^{1+\varepsilon}}{T} \bigg).
\end{equation*}
We deform the contour of integration to the path consisting of line segments connecting the points $1+\varepsilon -iT$, $\frac{1}{2} -iT$, $\frac{1}{2} +iT$, and $1+\varepsilon+iT$, leaving a residue from the pole of $\zeta(s)$ at $s=1$. We estimate the contribution of the horizontal line segments using the Lindel\"{o}f hypothesis and the Phragm\'{e}n-Lindel\"{o}f Theorem, and then insert the definitions \eqref{deltakdef} and \eqref{Ikdef} to deduce that
\begin{equation*}
\Delta_k(x) = \frac{1}{2\pi i } \int_{\frac{1}{2} -iY}^{\frac{1}{2}  +iY}\zeta^k(s) \frac{x^s}{s}\,ds + I_k(x; Y,T)  + O\Big(x^{\varepsilon} + x^{1+\varepsilon} T^{-1+\varepsilon} \Big).
\end{equation*}
We evaluate the integral on the right-hand side by deforming its contour of integration to the path consisting of line segments connecting the points $\frac{1}{2} -iY$, $-\varepsilon-iY$, $-\varepsilon+iY$, and $\frac{1}{2} +iY$, leaving a residue of size $O(1)$ from the pole of $1/s$ at $s=0$. We use the Lindel\"{o}f hypothesis, the functional equation, and the Phragm\'{e}n-Lindel\"{o}f Theorem to bound the contribution of the horizontal line segments, and arrive at
\begin{equation*}
\Delta_k(x) = \frac{1}{2\pi i } \int_{-\varepsilon-iY}^{-\varepsilon +iY}\zeta^k(s) \frac{x^s}{s}\,ds + I_k(x; Y,T)  + O\Big( x^{\varepsilon} Y^{\frac{k}{2}-1} + x^{\frac{1}{2}} Y^{-1+\varepsilon}  + x^{1+\varepsilon} T^{-1+\varepsilon}  \Big).
\end{equation*}
Lemma~\ref{lemma:lester2.5} now follows from this and Lemma~2.4 of Lester~\cite{lester}, which states that
\begin{align*}
\frac{1}{2\pi i } \int_{-\varepsilon-iY}^{-\varepsilon +iY}\zeta^k(s) \frac{x^s}{s}\,ds = Q_k( x; Y^k/(2\pi)^k ) + O\big( x^{\varepsilon} Y^{\frac{k}{2}-1} + x^{1+\varepsilon} Y^{-\frac{k}{2}-\frac{1}{2}} \big)
\end{align*}
for $Y\leq x$, where $Q_k$ is defined by \eqref{Qkdef}.
\end{proof}

We will bound the fourth moment of $\Delta_k(x)$ by applying the Erd\"{o}s-Tur\'{a}n inequality together with van der Corput's method for estimating exponential sums in a way similar to the proof of Lemma~4 of Tsang~\cite{tsang4thmoment}. This technique is embodied in the following lemma.

\begin{lem}\label{lemma:erdosturan1}
Let $\|x\|$ denote the distance from $x$ to the nearest integer. If $\rho>0$, $W\geq 1$, and $0<\alpha \ll W^{1/k}$, then
\begin{equation*}
\begin{split}
\# \big\{ \mu \in \mathbb{Z} \ : \ W<\mu\leq 2W  \mbox{\upshape{ and }} \big\| \big(\mu^{1/k} +\alpha \big)^k \big\| \leq \rho \big\}
\ll_k  W\rho  + W^{\frac{2}{3}-\frac{1}{3k}} \alpha^{1/3} + W^{\frac{1}{2}+\frac{1}{2k}} \alpha^{-1/2},
\end{split}
\end{equation*}
with the implied constant depending only on $k$.
\end{lem}
\begin{proof}
The Erd\"{o}s-Tur\'{a}n inequality (see, for example, \cite[Corollary~1.1]{montgomery}) implies that
\begin{equation}\label{erdosturan1}
\begin{split}
\# & \big\{ \mu \in \mathbb{Z} \ : \ W<\mu \leq 2W  \mbox{\upshape{ and }} \big\| \big(\mu^{1/k} +\alpha \big)^k \big\| \leq \rho \big\} \\
& \leq 2W\rho + \frac{W}{L+1} + 3 \sum_{\nu=1}^L \frac{1}{\nu} \Bigg|\sum_{W<\mu\leq 2W} e \Big(\nu\big(\mu^{1/k} +\alpha\big)^k  \Big)\Bigg|
\end{split}
\end{equation}
for every positive integer $L$. To estimate the exponential sum, let
\begin{equation}\label{fforlemmaerdosturan}
f(x) = \nu\big(x^{1/k} +\alpha\big)^k.
\end{equation}
Then
$$
f''(x) =  -\Big(1-\frac{1}{k}\Big) \nu \alpha  \big(x^{1/k} +\alpha \big)^{k-2}x^{\frac{1}{k}-2}.
$$
Thus, since $0<\alpha \ll W^{1/k}$, there are positive constants $A_k$ and $B_k$ that depend only on $k$ such that
$$
A_k \nu  \alpha W^{-1-\frac{1}{k}} \leq -f''(x) \leq B_k \nu \alpha W^{-1-\frac{1}{k}}
$$
whenever $W\leq x\leq 2W$. Hence van der Corput's method~\cite[Theorem~5.9]{titchmarsh} gives
$$
\sum_{W<\mu\leq 2W} e(f(\mu)) \ll_k \nu^{1/2} W^{\frac{1}{2}-\frac{1}{2k}}\alpha^{1/2}  + \nu^{-1/2} W^{\frac{1}{2}+\frac{1}{2k}} \alpha^{-1/2} .
$$
From this, the definition \eqref{fforlemmaerdosturan} of $f$, and \eqref{erdosturan1}, we arrive at
\begin{equation*}
\begin{split}
\# & \big\{ \mu\in \mathbb{Z} \ : \ W<\mu\leq 2W  \mbox{\upshape{ and }} \big\| \big(\mu^{1/k} +\alpha\big)^k \big\| \leq \rho \big\} \\
& \ll_k \ W\rho + \frac{W}{L} + L^{1/2} W^{\frac{1}{2}-\frac{1}{2k}} \alpha^{1/2} + W^{\frac{1}{2}+\frac{1}{2k}} \alpha^{-1/2}.
\end{split}
\end{equation*}
To complete the proof of the lemma, we optimize this bound and choose $L$ to be the least integer that is greater than $W^{\frac{1}{3}+\frac{1}{3k}}\alpha^{-1/3}$.
\end{proof}

\section{Lester's method}\label{section:lester}

In proving Theorem~\ref{thm:deltakfourthmmnt}, we will bound the fourth moment of $I_k(x;Y,T)$ by applying Lester's method together with the Riesz-Thorin Interpolation Theorem. In this section, let $1\leq Y\leq T\leq X$ and let $\Xi$ be the line segment from $\frac{1}{2}+iY$ to $\frac{1}{2}+iT$. We view $\Xi$ as a measure space in such a way that
$$
\int_{\Xi}f = -i\int_{\frac{1}{2}+iY}^{\frac{1}{2}+iT} f(s) \,ds =\int_Y^T f(\tfrac{1}{2}+it)\,dt
$$
for all continuous functions $f:\Xi\rightarrow \mathbb{C}$. Define the operator $\mathcal{T}$ by
\begin{equation}\label{operatorTdef}
\mathcal{T}f(x) = \frac{1}{\pi i } \int_{\frac{1}{2} +iY}^{\frac{1}{2}  +iT}f(s) x^s\,ds.
\end{equation}
Note that if $f\in L^p(\Xi)$ for some $p\geq 1$, then H\"{o}lder's inequality implies that $\mathcal{T}f(x)$ exists for all $x>0$, and that $\mathcal{T}f$ is continuous on $(0,\infty)$. Thus, if $f\in L^p(\Xi)$ for some $p\geq 1$, then $\mathcal{T} f \in L^q([X,2X])$ for all $q\geq 1$. In the next two lemmas, we use $\| f\|_p$ to denote the norm of $f$ in $L^p(\Xi)$, and we use $\|\mathcal{T} f\|_q$ to denote the norm of $\mathcal{T} f$ in $L^q([X,2X])$.

\begin{lem}\label{lemma:TL2toL2}
If $f\in L^2(\Xi)$, then $\|\mathcal{T}f\|_2 \ll X^{1+\varepsilon} \|f\|_2$. The implied constant here depends only on $\varepsilon$.
\end{lem}
\begin{proof}
Let $w:(0,\infty)\rightarrow \mathbb{R}$ be a nonnegative smooth function of compact support such that $w(u)=1$ whenever $1\leq u\leq 2$. Then
\begin{equation*}
\int_X^{2X} |\mathcal{T}f(x)|^2 \,dx \leq \int_0^{\infty} |\mathcal{T}f(x)|^2 w\Big( \frac{x}{X} \Big) \,dx.
\end{equation*}
We replace $\mathcal{T}f(x)$ on the right-hand side by its definition \eqref{operatorTdef}, expand the square, apply Fubini's theorem, and make a change of variables to arrive at
\begin{equation}\label{TL2toL2eqn1}
\int_X^{2X} |\mathcal{T}f(x)|^2 \,dx \leq \frac{X^2}{\pi^2}  \int_{Y}^{T}\int_{Y}^{T}f(\tfrac{1}{2}+it)    \overline{f(\tfrac{1}{2} +iv)}  X^{i(t-v)} \mathcal{J} (t-v)\,dv\,dt,
\end{equation}
where $\mathcal{J}(y):=\int_0^{\infty} u^{1+iy} w( u)  \,du$. Repeated integration by parts shows that $\mathcal{J}(y)\ll_A \min\{1,|y|^{-A}\}$ for arbitrarily large $A>0$. From this and the inequality $|ab|\ll |a|^2+|b|^2$, we deduce for any given $\eta>0$ that
\begin{equation}\label{TL2toL2eqn2}
\begin{split}
\mathop{\int_{Y}^{T}\int_{Y}^{T}}_{|t-v|>X^{\eta}}f(\tfrac{1}{2}+it)    \overline{f(\tfrac{1}{2} +iv)}  X^{i(t-v)} \mathcal{J} (t-v)\,dv\,dt
& \ll_{A,\eta} \frac{1}{X^A} \int_Y^T\int_Y^T|f(\tfrac{1}{2}+it)   |^2\,dv\,dt \\
& = \frac{T-Y}{X^A}\|f\|_2^2.
\end{split}
\end{equation}
On the other hand, the bound $\mathcal{J}(y)\ll 1$ and the inequality $|ab|\ll |a|^2+|b|^2$ imply that
\begin{equation*}
\begin{split}
\mathop{\int_{Y}^{T}\int_{Y}^{T}}_{|t-v|\leq X^{\eta}}f(\tfrac{1}{2}+it)    \overline{f(\tfrac{1}{2} +iv)}  X^{i(t-v)} \mathcal{J} (t-v)\,dv\,dt
\ll \mathop{\int_{Y}^{T}\int_{Y}^{T}}_{|t-v|\leq X^{\eta}}|f(\tfrac{1}{2}+it)   |^2\,dv\,dt
\leq X^{\eta} \|f\|_2^2.
\end{split}
\end{equation*}
From this, \eqref{TL2toL2eqn2}, \eqref{TL2toL2eqn1}, and the fact that $T-Y\leq X$, we arrive at
\begin{equation*}
\int_X^{2X} |\mathcal{T}f(x)|^2 \,dx \ll X^{2+\eta} \|f\|_2^2.
\end{equation*}
Taking the square root of both sides, we finish the proof upon choosing $\eta$ to be an arbitrarily small $\varepsilon>0$.
\end{proof}

\begin{lem}\label{lemma:TL4/3toL4}
If $f\in L^{4/3}(\Xi)$, then $\|\mathcal{T}f\|_4 \ll X^{\frac{3}{4}+\varepsilon} \|f\|_{4/3}$. The implied constant here depends only on $\varepsilon$.
\end{lem}
\begin{proof}
By taking the absolute value of the integrand on the right-hand side of \eqref{operatorTdef}, we see that $\|\mathcal{T} f\|_{\infty}\ll X^{1/2}\|f\|_1$ for all $f\in L^1(\Xi)$. Lemma~\ref{lemma:TL2toL2} states that $\|\mathcal{T} f\|_{2}\ll X^{1+\varepsilon} \|f\|_2$ for all $f\in L^2(\Xi)$. It follows from these and the Riesz-Thorin Interpolation Theorem (see, for example, \cite[p.~52]{steinshakarchi}) that
$$
\|\mathcal{T} f\|_{4} \ll \big( X^{1/2}\big)^{1/2}\big(X^{1+\varepsilon}\big)^{1/2} \|f\|_{4/3} = X^{\frac{3}{4}+\varepsilon} \|f\|_{4/3}
$$
for all $f\in L^{4/3}(\Xi)$.
\end{proof}

\begin{lem}\label{lemma:Ik4thmoment}
Assume the Lindel\"{o}f hypothesis. If $1\leq Y\leq T\leq X$ and $I_k$ is defined by \eqref{Ikdef}, then
\begin{equation*}
\int_X^{2X} \big| I_k(x;Y,T) \big|^4\,dx \ll \frac{X^{3+\varepsilon}}{Y}.
\end{equation*}
\end{lem}
\begin{proof}
Let $f(s)=s^{-1}\zeta(s)^k$. Then the definitions \eqref{Ikdef} and \eqref{operatorTdef} of $I_k$ and $\mathcal{T}$ imply that
\begin{equation*}
I_k(x;Y,T) =\text{Re}\big( \mathcal{T}f(x) \big).
\end{equation*}
From this, the inequality $|\text{Re}(z)|\leq |z|$, and Lemma~\ref{lemma:TL4/3toL4}, we arrive at
\begin{equation*}
\int_X^{2X} \big| I_k(x;Y,T) \big|^4\,dx \ll X^{3+\varepsilon}\Bigg(\int_Y^T \frac{|\zeta(\frac{1}{2}+it)|^{4k/3}}{t^{4/3}}\,dt \Bigg)^3.
\end{equation*}
The right-hand side is $\ll X^{3+\varepsilon}Y^{-1}$ if the Lindel\"{o}f hypothesis is true.
\end{proof}

We remark that the Lindel\"{o}f hypothesis is unnecessary for the case $k=3$ of Lemma~\ref{lemma:Ik4thmoment} because the size of the fourth moment of $\zeta(s)$ is known \cite[(7.6.2)]{titchmarsh}. Moreover, Lemma~\ref{lemma:lester2.5} may be made unconditional by using any $\delta$ satisfying \eqref{deltamomentbound}, as in Lemma~2.5 of Lester~\cite{lester}. These facts together with the arguments in Section~\ref{section:fourthmoment} lead to an unconditional proof of Theorem~\ref{thm:deltakfourthmmnt} for $k=3$. However, as mentioned earlier, the better bound \eqref{ivicbound} has been found by Ivi\'{c}~\cite{ivic}.

\section{The mean square of \texorpdfstring{$\Delta_k(x+h)-\Delta_k(x)$}{▲k(x+h)-▲k(x)} }\label{section:meansquare}

We now carry out Selberg's method~\cite{selberg} to prove Theorem~\ref{thm:dkinshortintervalsnosup}. Let $T\geq 2$ be a parameter, and define $\kappa>0$ by
\begin{equation}\label{kappadef}
e^{\kappa} = 1+\frac{1}{T}.
\end{equation}
Suppose that $0\leq\delta<1/2$ and $\delta$ satisfies \eqref{deltamomentbound}. Lemma~\ref{lemma:perron} then guarantees the existence of an increasing sequence $T_1,T_2,\dots$ of positive real numbers such that $\lim_{m\rightarrow \infty}T_m =\infty$ and
\begin{equation*}
\Delta_k^*(e^{\tau+\kappa})-\Delta_k^*(e^{\tau}) =\lim_{m\rightarrow \infty} \frac{1}{2\pi i} \int_{\frac{1}{2}+\delta-iT_m}^{\frac{1}{2}+\delta+iT_m} e^{s\tau}\left(\frac{e^{s\kappa}-1}{s} \right)\zeta^k(s) \,ds
\end{equation*}
for all real numbers $\tau$. Divide both sides by $\exp(\tau(\frac{1}{2}+\delta))$ and write the variable of integration $s$ as $\frac{1}{2}+\delta -2\pi it$ to arrive at
\begin{equation*}
\frac{\Delta_k^*(e^{\tau+\kappa})-\Delta_k^*(e^{\tau})}{e^{\tau(\frac{1}{2}+\delta)}}  =\lim_{m\rightarrow \infty} \int_{- T_m/(2\pi)}^{T_m/(2\pi)}e^{-2\pi i\tau t} \left(\frac{e^{\kappa(\frac{1}{2}+\delta-2\pi it)}-1}{\tfrac{1}{2}+\delta-2\pi it} \right) \zeta^k(\tfrac{1}{2}+\delta-2\pi it) \,dt
\end{equation*}
for all real $\tau$. The right-hand side is a Fourier transform, and we thus deduce from the Plancherel theorem that
\begin{equation*}
\int_{-\infty}^{\infty} \Bigg| \frac{\Delta_k^*(e^{\tau+\kappa})-\Delta_k^*(e^{\tau})}{e^{\tau(\frac{1}{2}+\delta)}} \Bigg|^2 \,d\tau = \int_{-\infty}^{\infty} \Bigg| \left(\frac{e^{\kappa(\frac{1}{2}+\delta-2\pi it)}-1}{\tfrac{1}{2}+\delta-2\pi it} \right) \zeta^k(\tfrac{1}{2}+\delta-2\pi it)\Bigg|^2\,dt.
\end{equation*}
We insert into this the definition \eqref{kappadef} and make the changes of variables $\tau\mapsto \log x$ and $t\mapsto -t/(2\pi )$ to arrive at
\begin{equation}\label{eqn:selberg}
\int_{0}^{\infty} \bigg| \Delta_k \left(x + \frac{x}{T}\right)-\Delta_k (x) \bigg|^2 \,\frac{dx}{x^{2+2\delta}}  = \frac{1}{ \pi}\int_{0}^{\infty} \Bigg| \left(\frac{e^{\kappa(\frac{1}{2}+\delta+ it)}-1}{\tfrac{1}{2}+\delta+ it} \right) \zeta^k(\tfrac{1}{2}+\delta+it)\Bigg|^2\,dt,
\end{equation}
where we also used the facts that $\zeta(\overline{s})=\overline{\zeta(s)}$ and $\Delta_k^*(x)=\Delta_k(x)$ for almost every $x$ by the definition \eqref{deltak*def} of $\Delta_k^*$.

To bound the right-hand side of \eqref{eqn:selberg} using moments of $\zeta(s)$, we split the interval of integration into dyadic parts. If $\ell$ is a nonnegative integer and $2^{\ell}T\leq t\leq 2^{\ell+1}T$, then the definition \eqref{kappadef} of $\kappa$ implies that $\exp(\kappa(\frac{1}{2}+\delta))\ll 1$ for $\delta<1/2$, and hence
\begin{equation}\label{eqn:expbound1}
\frac{e^{\kappa(\frac{1}{2}+\delta+ it)}-1}{\tfrac{1}{2}+\delta+ it} \ll \frac{1}{t} \ll \frac{1}{2^{\ell}T}.
\end{equation}
On the other hand, \eqref{kappadef} implies that $\kappa =\log (1+1/T)\leq 1/T$. Thus, if $0\leq t\leq T$ and $0\leq \delta<1/2$, then
\begin{equation}\label{eqn:expbound2}
\frac{e^{\kappa(\frac{1}{2}+\delta+ it)}-1}{\tfrac{1}{2}+\delta+ it} \ll \kappa \ll \frac{1}{T}
\end{equation}
because $e^z-1\ll |z|$ for $|z|\leq 2$. From \eqref{eqn:selberg}, \eqref{eqn:expbound1}, and \eqref{eqn:expbound2}, we deduce the following lemma.

\begin{lem}\label{lem:selberg}
Let $T\geq 2$. If $\,0\leq \delta<1/2$ and $\delta$ satisfies \eqref{deltamomentbound}, then  
\begin{align*}
\int_{0}^{\infty} & \bigg| \Delta_k \left(x + \frac{x}{T}\right)-\Delta_k (x) \bigg|^2\,\frac{dx}{x^{2+2\delta}} \\
& \ll \frac{1}{T^2}\int_0^T |\zeta(\tfrac{1}{2}+\delta+it)|^{2k}\,dt + \frac{1}{T^2} \sum_{\ell=0}^{\infty} \frac{1}{2^{2\ell}} \int_{2^{\ell}T}^{2^{\ell+1}T}  |\zeta(\tfrac{1}{2}+\delta+it)|^{2k}\,dt,
\end{align*}
with absolute implied constant.
\end{lem}

We now finish the proof of Theorem~\ref{thm:dkinshortintervalsnosup}. Let $T\geq 2$ and suppose that $\delta\geq 0$ satisfies \eqref{deltamomentbound}. Without loss of generality, we may assume that $\delta<1/2$ since reducing the value of $\delta$ improves the bound in the conclusion of Theorem~\ref{thm:dkinshortintervalsnosup}. Then Lemma~\ref{lem:selberg} and \eqref{deltamomentbound} imply
\begin{equation}\label{eqn:selberg2}
\int_{0}^{\infty} \bigg| \Delta_k \left(x + \frac{x}{T}\right)-\Delta_k (x) \bigg|^2\,\frac{dx}{x^{2+2\delta}} \ll \frac{1}{T^{1-\varepsilon}} + \frac{1}{T^{1-\varepsilon}} \sum_{\ell=0}^{\infty} \frac{1}{2^{\ell(1-\varepsilon)}}  \ll \frac{1}{T^{1-\varepsilon}},
\end{equation}
with implied constant depending only on the implied constant in \eqref{deltamomentbound}. Since the integrand in \eqref{eqn:selberg2} is nonnegative, we may truncate the integral to be over $[X,2X]$ and deduce that
\begin{equation*}
\int_{X}^{2X} \bigg| \Delta_k \left(x + \frac{x}{T}\right)-\Delta_k (x) \bigg|^2\,dx \ll \frac{X^{2+2\delta}}{T^{1-\varepsilon}}
\end{equation*}
for all $X>0$. Replacing $X$ by $X/2$, $X/4$, $X/8$,$\dots$, and adding the results leads to
\begin{equation*}
\int_{0}^{X} \bigg| \Delta_k \left(x + \frac{x}{T}\right)-\Delta_k (x) \bigg|^2\,dx \ll \frac{X^{2+2\delta}}{T^{1-\varepsilon}}.
\end{equation*}
We relabel $T$ as $1/\beta$ and arrive at
\begin{equation*}
\int_{0}^{X} |\Delta_k \left(x + \beta x\right)-\Delta_k (x) |^2\,dx \ll \beta^{1-\varepsilon} X^{2+2\delta} 
\end{equation*}
for all $\beta$ in the interval $[0,1/2]$. From this and Lemma~\ref{lem:saffarivaughan}, we see that if $X>0$ and $0< h\leq X/16$, then
\begin{align*}
\int_{X/2}^X |\Delta_k(x+h)-\Delta_k(x)|^2\,dx
& \ll \frac{X}{h} \int_0^{8h/X} \int_{0}^{X} |\Delta_k \left(x + \beta x\right)-\Delta_k (x) |^2\,dx\,d\beta \\
& \ll \frac{X}{h} \int_0^{8h/X} \beta^{1-\varepsilon} X^{2+2\delta} \,d\beta \\
& \ll h^{1-\varepsilon} X^{1+2\delta+\varepsilon},
\end{align*}
with implied constant depending only on the implied constant in \eqref{deltamomentbound}. Replacing $X$ by $2X$ completes the proof of Theorem~\ref{thm:dkinshortintervalsnosup}.

Corollary~\ref{cor:d3inshortintervals} follows from Theorem~\ref{thm:dkinshortintervalsnosup} and the theorem of Heath-Brown~\cite{heathbrown} (see also \S 7.22 of \cite{titchmarsh}) that implies that if $k=3$ then $\delta=1/12$ satisfies \eqref{deltamomentbound}. If the Lindel\"{o}f hypothesis is true, then $\delta=0$ satisfies \eqref{deltamomentbound}, and so Corollary~\ref{cor:dkinshortintervalsLH} holds.

Having proved Theorem~\ref{thm:dkinshortintervalsnosup} and its corollaries, we next prove Theorem~\ref{thm:dkinshortintervalsRH}. Assuming the Riemann hypothesis, Harper~\cite{harper} has shown that
\begin{equation*}
\int_0^{\tau} |\zeta(\tfrac{1}{2}+it)|^{2k} \,dt \ll_k \tau(\log \tau)^{k^2}
\end{equation*}
for all $\tau\geq 2$. This, of course, implies that $\delta=0$ satisfies \eqref{deltamomentbound}. From these and Lemma~\ref{lem:selberg} with $\delta=0$, we deduce that if the Riemann hypothesis is true, then
\begin{equation*}
\int_{0}^{\infty} \bigg| \Delta_k \left(x + \frac{x}{T}\right)-\Delta_k (x) \bigg|^2\,\frac{dx}{x^{2}}  \ll_k \frac{(\log T)^{k^2}}{T }  + \frac{1}{T } \sum_{\ell=0}^{\infty} \frac{(\log(2^{\ell+1}T))^{k^2}}{2^{ \ell}}  \ll \frac{(\log T)^{k^2}}{T }
\end{equation*}
for all $T\geq 2$. Truncating the integral to be over $[X,2X]$, we arrive at
\begin{equation*}
\int_{X}^{2X} \bigg| \Delta_k \left(x + \frac{x}{T}\right)-\Delta_k (x) \bigg|^2\,dx \ll  \frac{X^2(\log T)^{k^2}}{T }
\end{equation*}
for all $X>0$. Replacing $X$ by $X/2$, $X/4$, $X/8$,$\dots$, and adding the results leads to
\begin{equation*}
\int_{0}^{X} \bigg| \Delta_k \left(x + \frac{x}{T}\right)-\Delta_k (x) \bigg|^2\,dx \ll  \frac{X^2(\log T)^{k^2}}{T }.
\end{equation*}
We relabel $T$ as $1/\beta$ and arrive at
\begin{equation*}
\int_0^X |\Delta_k \left(x + \beta x\right)-\Delta_k (x) |^2 \,dx  \ll_k  X^{2 } \beta |\log \beta|^{k^2}
\end{equation*}
for all $\beta$ in the interval $(0,1/2]$. From this and Lemma~\ref{lem:saffarivaughan}, we see that if $X>0$ and $0< h\leq X/16$, then
\begin{align*}
\int_{X/2}^X |\Delta_k(x+h)-\Delta_k(x)|^2\,dx
& \ll \frac{X}{h} \int_0^{8h/X} \int_{0}^{X} |\Delta_k \left(x + \beta x\right)-\Delta_k (x) |^2\,dx\,d\beta \\
& \ll \frac{X}{h} \int_0^{8h/X} X^{2 } \beta |\log \beta|^{k^2} \,d\beta.
\end{align*}
We may evaluate the latter integral via repeated integration by parts, which leads to
\begin{equation*}
\int_{X/2}^X |\Delta_k(x+h)-\Delta_k(x)|^2\,dx \ll_k hX \left( \log \left( \frac{X}{8h}\right)\right)^{k^2}
\end{equation*}
for $0< h\leq X/16$. Replacing $X$ by $2X$ completes the proof of Theorem~\ref{thm:dkinshortintervalsRH}.

\section{The fourth moment of \texorpdfstring{$\Delta_k(x)$}{▲k(x)}}\label{section:fourthmoment}

In this section, we shall prove Theorem~\ref{thm:deltakfourthmmnt}. Suppose that $1\leq Y\leq T\leq X$. We apply Lemma~\ref{lemma:lester2.5} and use the inequality $|a+b|^4\ll |a|^4+|b|^4$ to write
\begin{equation*}
\begin{split}
& \frac{1}{X}\int_X^{2X} |\Delta_k(x)|^4\,dx \\
& \ll \frac{1}{X}\int_X^{2X} |Q_k( x; Y^k/(2\pi)^k )|^4\,dx  + \frac{1}{X}\int_X^{2X} |I_k(x;Y,T) |^4\,dx + \frac{1}{X}\int_X^{2X} |E_k(x;Y,T) |^4\,dx .
\end{split}
\end{equation*}
From this, \eqref{Ekbound}, and Lemma~\ref{lemma:Ik4thmoment}, we deduce that
\begin{equation}\label{4thmmntsplit}
\begin{split}
\frac{1}{X}\int_X^{2X} |\Delta_k(x)|^4\,dx \ll \frac{1}{X}\int_X^{2X}
& |Q_k( x; Y^k/(2\pi)^k )|^4\,dx + \frac{X^{2+\varepsilon}}{Y}\\
& + \frac{X^{4+\varepsilon}}{ Y^{2k+2}} + X^{\varepsilon}Y^{2k-4} + \frac{X^{2+\varepsilon}}{Y^4} + \frac{X^{4+\varepsilon}}{T^4}
\end{split}
\end{equation}
under the assumption of the Lindel\"{o}f hypothesis. To prove Theorem~\ref{thm:deltakfourthmmnt}, our main task in this section is to bound the first term on the right-hand side of \eqref{4thmmntsplit}. For brevity, in this section we set
\begin{equation}\label{Vdef}
V : = \left(\frac{Y}{2\pi}\right)^k,
\end{equation}
\begin{equation}\label{a1def}
a_1 = a_1(\mu,\nu,m,n;k) := \frac{d_k(\mu)d_k(\nu)d_k(m)d_k(n)}{(\mu\nu m n)^{\frac{1}{2}+\frac{1}{2k}}},
\end{equation}
and
\begin{equation}\label{X1def}
X_1=X_1(\mu,\nu,m,n;V,X):= \min\{2X,V/\mu,V/\nu,V/m,V/n\}\leq 2X.
\end{equation}
Use the definition \eqref{Qkdef} of $Q_k$, interchange the order of summation, and repeatedly apply the trigonometric identity $2\cos a \cos b =\cos(a+b) + \cos(a-b)$ to write
\begin{equation}\label{Qsplit}
\frac{1}{X}\int_X^{2X} |Q_k( x; V )|^4\,dx = \frac{1}{\pi^4k^2} \bigg( \frac{3}{8}S_1 + \frac{1}{2}S_2 + \frac{1}{8}S_3 \bigg),
\end{equation}
where $S_1$, $S_2$, and $S_3$ are defined by
\begin{equation}\label{S1def}
S_1 := \frac{1}{ X}  \sum_{\mu,\nu,m,n \leq V/X} a_1 \int_X^{X_1} x^{2-\frac{2}{k}} \cos\Big(2\pi k x^{1/k} \big(\mu^{1/k} +\nu^{1/k} - m^{1/k}-n^{1/k} \big)  \Big)\,dx,
\end{equation}
\begin{equation}\label{S2def}
S_2 := \frac{1}{ X} \sum_{\mu,\nu,m,n \leq V/X} a_1 \int_X^{X_1} x^{2-\frac{2}{k}} \cos\bigg(2\pi k x^{1/k} \big(\mu^{1/k} +\nu^{1/k} + m^{1/k}-n^{1/k} \big) +\frac{(k-3)\pi}{2} \bigg)\,dx,
\end{equation}
and
\begin{equation}\label{S3def}
S_3 := \frac{1}{ X} \sum_{\mu,\nu,m,n \leq V/X} a_1 \int_X^{X_1} x^{2-\frac{2}{k}} \cos\Big(2\pi k x^{1/k} \big(\mu^{1/k} +\nu^{1/k} + m^{1/k}+n^{1/k} \big) + (k-3)\pi  \Big)\,dx,
\end{equation}
where the summation indices $\mu,\nu,m,n$ run through positive integers.

Our first task is to estimate $S_1$, which is defined by \eqref{S1def}. We bound the right-hand side of \eqref{S1def} by taking the absolute value of each term. By symmetry, we may then assume without loss of generality that $\nu\leq \mu$, $n\leq m$, and $n\leq \nu$. We thus arrive at
\begin{equation}\label{S1split}
S_1 \ll S_{11}+S_{12},
\end{equation}
where $S_{11}$ and $S_{12}$ are defined by
\begin{equation*}
S_{11} : = \frac{1}{ X} \sum_{\substack{\mu,\nu,m,n \leq V/X \\ \nu\leq \mu \\ n\leq m \\ n=\nu }} a_1 \Bigg| \int_X^{X_1} x^{2-\frac{2}{k}} \cos\Big(2\pi k x^{1/k} \big(\mu^{1/k}  - m^{1/k}  \big)  \Big)\,dx \Bigg|
\end{equation*}
and
\begin{equation}\label{S12def}
S_{12} : = \frac{1}{ X} \sum_{\substack{\mu,\nu,m,n \leq V/X \\ \nu\leq \mu \\ n\leq m \\ n<\nu }} a_1 \Bigg| \int_X^{X_1} x^{2-\frac{2}{k}} \cos\Big(2\pi k x^{1/k} \big(\mu^{1/k} +\nu^{1/k} - m^{1/k}-n^{1/k} \big)  \Big)\,dx \Bigg|.
\end{equation}

To bound $S_{11}$, we further write
\begin{equation}\label{S11split}
S_{11}=S_{111}+S_{112},
\end{equation}
where $S_{111}$ is the part of $S_{11}$ with $m=\mu$ and $S_{112}$ is the part with $m\neq \mu$. Using the definitions \eqref{a1def} and \eqref{X1def}, we deduce that
\begin{equation}\label{S111bound}
S_{111} = \frac{1}{ X} \sum_{\substack{\mu,\nu,m,n \leq V/X \\ \nu\leq \mu \\ n\leq m \\ n=\nu \\ m=\mu}} a_1 \int_X^{X_1} x^{2-\frac{2}{k}} \,dx \leq \frac{1}{ X} \sum_{\substack{ m,n \leq V/X \\ n\leq m }} \frac{  d_k^2(m)d_k^2(n)}{(  m n)^{1+\frac{1}{ k}}} \int_X^{2X} x^{2-\frac{2}{k}} \,dx \ll X^{2-\frac{2}{k}}.
\end{equation}
On the other hand, to bound $S_{112}$, we may assume without loss of generality that $m<\mu$, and integrate by parts to arrive at
\begin{equation*}
S_{112} \ll X^{2-\frac{3}{k}}\sum_{\substack{\mu,\nu,m,n \leq V/X \\ \nu\leq \mu \\ n\leq m \\ n=\nu \\ m<\mu}} \frac{a_1 }{\mu^{1/k}-m^{1/k}} = X^{2-\frac{3}{k}}\sum_{\substack{\mu,m,n \leq V/X \\ n\leq \mu \\ n\leq m \\ m<\mu}} \frac{d_k(\mu)d_k(m)d_k^2(n)}{n^{1+\frac{1}{k}} (\mu m)^{\frac{1}{2}+\frac{1}{2k}}\big(\mu^{1/k}-m^{1/k}\big)} .
\end{equation*}
Since $\mu^{1/k}-m^{1/k}\gg (\mu-m)\mu^{\frac{1}{k}-1} $ for $\mu>m$ and $d_k(j)\ll j^{\varepsilon}$ for all positive integers $j$, it follows that
\begin{equation*}
S_{112} \ll X^{2-\frac{3}{k}} V^{\varepsilon} \sum_{\substack{\mu,m,n \leq V/X \\ n\leq \mu \\ n\leq m \\ m<\mu}} \frac{\mu^{\frac{1}{2}-\frac{3}{2k}} }{n^{1+\frac{1}{k}} m^{\frac{1}{2}+\frac{1}{2k}}(\mu-m)}.
\end{equation*}
The $m$-sum here is $O(1)$ by the Cauchy-Schwarz inequality, and so
\begin{equation*}
S_{112} \ll X^{2-\frac{3}{k}} V^{\varepsilon} \left( \frac{V}{X}\right)^{\frac{3}{2}-\frac{3}{2k}} = X^{\frac{1}{2}-\frac{3}{2k}} V^{\frac{3}{2}-\frac{3}{2k}+\varepsilon}.
\end{equation*}
It follows from this, \eqref{S111bound}, and \eqref{S11split} that
\begin{equation}\label{S11bound}
S_{11} \ll X^{2-\frac{2}{k}} + X^{\frac{1}{2}-\frac{3}{2k}} V^{\frac{3}{2}-\frac{3}{2k}+\varepsilon}.
\end{equation}

Having estimated $S_{11}$, we next bound $S_{12}$, which is defined by \eqref{S12def}. Let $\xi>0$ be a parameter, to be chosen later, such that $\xi<1$ and
\begin{equation}\label{xicondition}
\xi \left( \frac{V}{X}\right)^{1-\frac{1}{k}} =o(1)
\end{equation}
as $X\rightarrow \infty$. Define $\Lambda_1$ by
\begin{equation}\label{Lambda1def}
\Lambda_1= \Lambda_1(\mu,\nu,m,n;k) := \mu^{1/k} +\nu^{1/k} - m^{1/k}-n^{1/k}.
\end{equation}
Split the sum $S_{12}$, defined by \eqref{S12def}, and write
\begin{equation}\label{S12split}
S_{12} = S_{121}+S_{122},
\end{equation}
where $S_{121}$ is the part with $|\Lambda_1|\leq \xi$ and $S_{122}$ is the part with $|\Lambda_1|>\xi$.

To estimate $S_{121}$, we bound the integral in \eqref{S12def} trivially using \eqref{X1def}, and then use \eqref{a1def} to deduce that
\begin{equation*}
S_{121} \ll X^{2-\frac{2}{k}} \sum_{\substack{\mu,\nu,m,n \leq V/X \\ \nu\leq \mu \\ n\leq m \\ n<\nu \\ |\Lambda_1|\leq \xi}} a_1 \ll X^{2-\frac{2}{k}}V^{\varepsilon} \sum_{\substack{\mu,\nu,m,n \leq V/X \\ \nu\leq \mu \\ n\leq m \\ n<\nu \\ |\Lambda_1|\leq \xi}} \frac{1}{(\mu\nu m n)^{\frac{1}{2}+\frac{1}{2k}}}.
\end{equation*}
Note that the summation conditions imply that $\mu>1$. We partition the range of the summation variable $\mu$ into dyadic intervals $(1,2],$ $(2,4],$ $(4,8],\dots$ to write
\begin{equation}\label{S121partition}
S_{121} \ll X^{2-\frac{2}{k}}V^{\varepsilon} \sum_M \sum_{\substack{\mu,\nu,m,n \leq V/X \\ M<\mu \leq 2M \\ \nu\leq \mu \\ n\leq m \\ n<\nu \\ |\Lambda_1|\leq \xi}} \frac{1}{(\mu\nu m n)^{\frac{1}{2}+\frac{1}{2k}}},
\end{equation}
where $M\geq 1$ runs through the powers of $2$ less than or equal to $V/X$. Our assumption that $\xi<1$, the definition \eqref{Lambda1def} of $\Lambda_1$, and the conditions satisfied by the summation variables in \eqref{S121partition} imply that $\nu,n,m \ll\mu \ll M$. It follows from this and the polynomial identity $x^k-y^k=(x-y)(x^{k-1}+x^{k-2}y+\cdots+y^{k-1})$ that
\begin{equation}\label{Lambda1identity}
\Big|\big(\mu^{1/k} +\nu^{1/k} - n^{1/k}\big)^k -m \Big| \ll_k |\Lambda_1| \mu^{1-\frac{1}{k}} \ll \xi M^{1-\frac{1}{k}}.
\end{equation}
From this, \eqref{xicondition}, and the fact that $M\leq V/X$, we see for large enough $X$ that, for each triple $\mu,\nu,n$ in \eqref{S121partition}, there is at most one integer $m$ such that $|\Lambda_1|\leq \xi$, and such an $m$ must satisfy
\begin{equation*}
m \asymp \big(\mu^{1/k} +\nu^{1/k} - n^{1/k}\big)^k \asymp \mu
\end{equation*}
because $\nu>n$. Furthermore, if such an $m$ exists, then it follows from \eqref{Lambda1identity} that
\begin{equation}\label{Lambda1distance}
\Big\|\big(\mu^{1/k} +\nu^{1/k} - n^{1/k}\big)^k \Big\| \ll_k \xi M^{1-\frac{1}{k}},
\end{equation}
where $\|x\|$ denotes the distance from $x$ to the nearest integer. These and \eqref{S121partition} imply that
\begin{equation*}
S_{121} \ll X^{2-\frac{2}{k}}V^{\varepsilon} \sum_M \sum_{\substack{\mu,\nu,n \leq V/X \\ M<\mu \leq 2M \\ \nu\leq \mu \\ n<\nu \\ \mbox{\scriptsize{\eqref{Lambda1distance}}}}} \frac{1}{(\nu  n)^{\frac{1}{2}+\frac{1}{2k}} \mu^{1+\frac{1}{k}}}.
\end{equation*}
From this and Lemma~\ref{lemma:erdosturan1} with $W=M$, $\rho=O_k( \xi M^{1-\frac{1}{k}})$, and $\alpha= \nu^{1/k} - n^{1/k}$, we arrive at
\begin{equation}\label{S121erdosturan}
\begin{split}
S_{121} \ll X^{2-\frac{2}{k}}V^{\varepsilon} \sum_M \frac{1}{M^{1+\frac{1}{k}}}\sum_{\substack{\nu,n \leq V/X \\ \nu\leq 2M \\ n<\nu }} \frac{1}{(\nu  n)^{\frac{1}{2}+\frac{1}{2k}}} \Big(\xi M^{2-\frac{1}{k}}  + M^{\frac{2}{3}-\frac{1}{3k}} \big( \nu^{1/k} - n^{1/k}\big)^{1/3} \\
+ M^{\frac{1}{2}+\frac{1}{2k}} \big( \nu^{1/k} - n^{1/k}\big)^{-1/2}\Big).
\end{split}
\end{equation}
Recall that, as in \eqref{S121partition}, $M$ runs through the powers of $2$ in the interval $[1,V/X]$. Thus
\begin{equation}\label{S1211}
\sum_M \frac{1}{M^{1+\frac{1}{k}}}\sum_{\substack{\nu,n \leq V/X \\ \nu\leq 2M \\ n<\nu }} \frac{1}{(\nu  n)^{\frac{1}{2}+\frac{1}{2k}}} \Big(\xi M^{2-\frac{1}{k}} \Big) \ll \xi \sum_M M^{1-\frac{2}{k}} \sum_{\nu \leq 2M}  \frac{1}{\nu^{1/k}} \ll \xi\left( \frac{V}{X}\right)^{2-\frac{3}{k}}.
\end{equation}
Similarly, since $\big( \nu^{1/k} - n^{1/k}\big)^{1/3} \leq \nu^{1/(3k)}$, we have
\begin{equation}\label{S1212}
\sum_M \frac{1}{M^{1+\frac{1}{k}}}\sum_{\substack{\nu,n \leq V/X \\ \nu\leq 2M \\ n<\nu }} \frac{1}{(\nu  n)^{\frac{1}{2}+\frac{1}{2k}}} \Big( M^{\frac{2}{3}-\frac{1}{3k}} \big( \nu^{1/k} - n^{1/k}\big)^{1/3}\Big)
\ll \left( \frac{V}{X}\right)^{\frac{2}{3}-\frac{2}{k}} \log V
\end{equation}
(the factor $\log V$ is necessary only when $k=3$). To estimate the contribution of the term with $\big( \nu^{1/k} - n^{1/k}\big)^{-1/2}$ in \eqref{S121erdosturan}, we use the bound $\nu^{1/k} - n^{1/k} \gg (\nu-n)\nu^{\frac{1}{k}-1}$ to deduce that
\begin{equation*}
\begin{split}
\sum_{n<\nu} \frac{\big( \nu^{1/k} - n^{1/k}\big)^{-1/2}}{n^{\frac{1}{2}+\frac{1}{2k}}}
=\sum_{n<\nu/2} + \sum_{\nu/2<n<\nu}
& \ll  \frac{1}{\nu^{1/(2k)}} \sum_{n<\nu/2}\frac{1}{n^{\frac{1}{2}+\frac{1}{2k}}} + \frac{1}{\nu^{1/k}} \sum_{\nu/2<n<\nu} (\nu-n)^{-1/2} \\
& \ll \nu^{\frac{1}{2}-\frac{1}{k}}.
\end{split}
\end{equation*}
Hence
\begin{equation*}
\begin{split}
\sum_M \frac{1}{M^{1+\frac{1}{k}}}\sum_{\substack{\nu,n \leq V/X \\ \nu\leq 2M \\ n<\nu }} \frac{1}{(\nu  n)^{\frac{1}{2}+\frac{1}{2k}}}
& \Big( M^{\frac{1}{2}+\frac{1}{2k}} \big( \nu^{1/k} - n^{1/k}\big)^{-1/2}\Big)
\ll \sum_M \frac{1}{M^{\frac{1}{2}+\frac{1}{2k}}} \sum_{\nu\leq 2M} \frac{1}{\nu^{3/(2k)}} \\
& \ll \sum_M M^{\frac{1}{2}-\frac{2}{k}}
\ll \max\Big\{ \log V, (V/X)^{\frac{1}{2} -\frac{2}{k}}\Big\}.
\end{split}
\end{equation*}
From this, \eqref{S1212}, \eqref{S1211}, and \eqref{S121erdosturan}, we arrive at
\begin{equation}\label{S121puttogether}
S_{121} \ll X^{2-\frac{2}{k}}V^{\varepsilon}\bigg( \xi\left( \frac{V}{X}\right)^{2-\frac{3}{k}} + \left( \frac{V}{X}\right)^{\frac{2}{3}-\frac{2}{k}} \log V + \max\Big\{ \log V, (V/X)^{\frac{1}{2} -\frac{2}{k}}\Big\}\bigg) .
\end{equation}
We may assume that $V\geq X$ since otherwise $S_{121}=0$ by \eqref{S121partition}. Thus $(V/X)^{\frac{1}{2} -\frac{2}{k}} \leq (V/X)^{\frac{2}{3}-\frac{2}{k}}$, and \eqref{S121puttogether} simplifies to
\begin{equation}\label{S121bound}
S_{121} \ll \xi X^{1/k} V^{2-\frac{3}{k}+\varepsilon} + X^{4/3} V^{\frac{2}{3}-\frac{2}{k}+\varepsilon}.
\end{equation}

Having bounded the sum $S_{121}$ in \eqref{S12split}, we next estimate $S_{122}$, which is the part of \eqref{S12def} that has $|\Lambda_1|>\xi$. Recalling the definitions \eqref{X1def} of $X_1$ and \eqref{Lambda1def} of $\Lambda_1$, we estimate the integral in \eqref{S12def} via integration by parts and then use \eqref{a1def} to arrive at
\begin{equation}\label{S122IBP}
S_{122} \ll X^{2-\frac{3}{k}}\sum_{\substack{\mu,\nu,m,n \leq V/X \\ \nu\leq \mu \\ n\leq m \\ n<\nu \\ |\Lambda_1|>\xi }} \frac{a_1}{|\Lambda_1|} \ll X^{2-\frac{3}{k}}V^{\varepsilon} \sum_{\substack{\mu,\nu,m,n \leq V/X \\ \nu\leq \mu \\ n\leq m \\ n<\nu \\ |\Lambda_1|>\xi }} \frac{1}{(\mu\nu mn)^{\frac{1}{2}+\frac{1}{2k}} |\Lambda_1|}.
\end{equation}
We split the range of $|\Lambda_1|$ dyadically to deduce from \eqref{S122IBP} that
\begin{equation}\label{S122dyadic}
S_{122} \ll X^{2-\frac{3}{k}} V^{\varepsilon} \sum_{L>\xi/2} \frac{1}{L} \sum_{\substack{\mu,\nu,m,n \leq V/X \\ \nu\leq \mu \\ n\leq m \\ n<\nu \\ L<|\Lambda_1|\leq 2L }} \frac{1}{(\mu\nu mn)^{\frac{1}{2}+\frac{1}{2k}} },
\end{equation}
where $L$ runs through the numbers $2^j$ with $j\in \mathbb{Z}$. Now if $n<\nu \leq \mu$ and $\Lambda_1 \ll \mu^{1/k}$, then the definition \eqref{Lambda1def} of $\Lambda_1$ and the binomial theorem imply
\begin{equation}\label{mbinomial}
m = \big( \mu^{1/k}+\nu^{1/k}-n^{1/k}-\Lambda_1\big)^k =  \big( \mu^{1/k}+\nu^{1/k}-n^{1/k}\big)^k + O_k\big( |\Lambda_1|\mu^{1-\frac{1}{k}}\big).
\end{equation}
Let $\varepsilon_k>0$ be a small enough constant, depending only on $k$, such that if $|\Lambda_1|\leq 2\varepsilon_k\mu^{1/k}$, then the error term in \eqref{mbinomial} has absolute value $\leq \mu/2$. Split the $L$-sum in \eqref{S122dyadic} and write
\begin{equation}\label{S122split}
S_{122} \ll \Sigma_1 + \Sigma_2,
\end{equation}
where $\Sigma_1$ is the part with $L\leq \varepsilon_k \mu^{1/k}$ and $\Sigma_2$ is the part with $L> \varepsilon_k \mu^{1/k}$. To bound $\Sigma_1$, observe that if $n<\nu \leq \mu$ and $L\leq \varepsilon_k \mu^{1/k}$, then \eqref{mbinomial} implies that there are at most $1 + O_k ( L \mu^{1-\frac{1}{k}} )$ integers $m$ satisfying $|\Lambda_1|\leq 2L$. Moreover, each such $m$ satisfies $m\asymp \mu$ by \eqref{mbinomial}, the definition of $\varepsilon_k$ below \eqref{mbinomial}, and the fact that $n<\nu \leq \mu$. Thus
\begin{equation}\label{Sigma1bound}
\Sigma_1 \ll X^{2-\frac{3}{k}} V^{\varepsilon}  \sum_{\substack{\mu,\nu,n \leq V/X \\ \nu\leq \mu \\ n<\nu }} \frac{1}{(\nu n)^{\frac{1}{2}+\frac{1}{2k}} \mu^{1+\frac{1}{k}}} \sum_{\xi/2<L\leq \varepsilon_k \mu^{1/k}} \frac{1}{L}\Big( 1 + O_k\big( L\mu^{1-\frac{1}{k}}\big)\Big).
\end{equation}
Recall that, as in \eqref{S122dyadic}, $L$ runs through powers of $2$. Thus the number of terms in the $L$-sum in \eqref{Sigma1bound} is $\ll V^{\varepsilon} |\log \xi|$, and so
\begin{equation}\label{Sigma1bound2}
\begin{split}
\Sigma_1
& \ll X^{2-\frac{3}{k}} V^{\varepsilon} |\log \xi|  \sum_{\substack{\mu,\nu,n \leq V/X \\ \nu\leq \mu \\ n<\nu }} \frac{1}{(\nu n)^{\frac{1}{2}+\frac{1}{2k}} \mu^{1+\frac{1}{k}}} \bigg( \frac{1}{\xi} + \mu^{1-\frac{1}{k}}\bigg) \\
& \ll X^{2-\frac{3}{k}} V^{\varepsilon} |\log \xi|  \sum_{\substack{\mu,\nu \leq V/X \\ \nu\leq \mu }} \frac{1}{\nu^{1/k} \mu^{1+\frac{1}{k}}} \bigg( \frac{1}{\xi} + \mu^{1-\frac{1}{k}}\bigg)\\
& \ll \xi^{-1} X^{1-\frac{1}{k}}V^{1-\frac{2}{k}+\varepsilon} |\log \xi| + V^{2-\frac{3}{k}+\varepsilon} |\log \xi|.
\end{split}
\end{equation}
To bound the sum $\Sigma_2$ in \eqref{S122split}, ignore the conditions $L<|\Lambda_1|\leq 2L$ and $n\leq m$, and then evaluate the $L$-sum as a geometric series to deduce that
\begin{equation*}
\begin{split}
\Sigma_2
& \ll  X^{2-\frac{3}{k}} V^{\varepsilon}  \sum_{\substack{\mu,\nu,m,n \leq V/X \\ \nu\leq \mu \\  n<\nu  }} \frac{1}{(\mu\nu mn)^{\frac{1}{2}+\frac{1}{2k}} }\sum_{L>\varepsilon_k \mu^{1/k}} \frac{1}{L} \ll X^{2-\frac{3}{k}} V^{\varepsilon}  \sum_{\substack{\mu,\nu,m,n \leq V/X \\ \nu\leq \mu \\  n<\nu  }} \frac{1}{(\nu mn)^{\frac{1}{2}+\frac{1}{2k}} \mu^{\frac{1}{2}+\frac{3}{2k}}} \\
& \ll X^{2-\frac{3}{k}} V^{\varepsilon}  \sum_{\substack{\mu,\nu,m \leq V/X \\ \nu\leq \mu   }} \frac{1}{\nu^{1/k} m^{\frac{1}{2}+\frac{1}{2k}}  \mu^{\frac{1}{2}+\frac{3}{2k}}} \ll X^{2-\frac{3}{k}} V^{\varepsilon}  \sum_{\mu,m\leq V/X} \frac{\mu^{\frac{1}{2}-\frac{5}{2k}}}{ m^{\frac{1}{2}+\frac{1}{2k}}} \ll V^{2-\frac{3}{k}+\varepsilon}.
\end{split}
\end{equation*}
From this, \eqref{Sigma1bound2}, and \eqref{S122split}, we arrive at
\begin{equation}\label{S122bound}
S_{122} \ll \xi^{-1} X^{1-\frac{1}{k}}V^{1-\frac{2}{k}+\varepsilon} |\log \xi| + V^{2-\frac{3}{k}+\varepsilon} (1+|\log \xi|).
\end{equation}

This, \eqref{S121bound}, \eqref{S12split}, \eqref{S11bound}, and \eqref{S1split} now imply
\begin{equation}\label{S1bound}
\begin{split}
S_1 \ll X^{2-\frac{2}{k}} + X^{\frac{1}{2}-\frac{3}{2k}} V^{\frac{3}{2}-\frac{3}{2k}+\varepsilon} + \xi X^{1/k} V^{2-\frac{3}{k}+\varepsilon} + X^{4/3} V^{\frac{2}{3}-\frac{2}{k}+\varepsilon} \\
+ \xi^{-1} X^{1-\frac{1}{k}}V^{1-\frac{2}{k}+\varepsilon} |\log \xi| + V^{2-\frac{3}{k}+\varepsilon} (1+|\log \xi|).
\end{split}
\end{equation}
This completes our estimation of $S_1$.

Our next task is to bound $S_2$, which is defined by \eqref{S2def}. The procedure is similar to our estimation of $S_{12}$, which starts with \eqref{S12split}, and so we only present a sketch. Define $\Lambda_2$ by
\begin{equation}\label{Lambda2def}
\Lambda_2= \Lambda_2(\mu,\nu,m,n;k) := \mu^{1/k} +\nu^{1/k} + m^{1/k}-n^{1/k}
\end{equation}
and let $\xi$ be as in \eqref{xicondition}. Split the sum $S_2$ in \eqref{S2def} to write
\begin{equation}\label{S2split}
S_2=S_{21}+S_{22},
\end{equation}
where $S_{21}$ is the part with $|\Lambda_2|\leq \xi$ and $S_{22}$ is the part with $|\Lambda_2|>\xi$. To bound $S_{21}$, we may assume that $m\leq \nu\leq \mu$. We bound the integral trivially and partition the range of $\mu$ into dyadic intervals to deduce that, similarly to \eqref{S121partition}, we have
\begin{equation*}
S_{21} \ll X^{2-\frac{2}{k}}V^{\varepsilon} \sum_M \sum_{\substack{\mu,\nu,m,n \leq V/X \\ M<\mu \leq 2M \\ m\leq \nu\leq \mu \\ |\Lambda_2|\leq \xi}} \frac{1}{(\mu\nu m n)^{\frac{1}{2}+\frac{1}{2k}}},
\end{equation*}
where $M$ runs through the powers of $2$ in the interval $[1/2,V/X]$. For each triple $m,\nu,\mu$ in this sum, the condition \eqref{xicondition} ensures that there is at most one integer $n$ such that $|\Lambda_2|\leq \xi$, and such an $n$ satisfies $n\asymp \mu$. If such an $n$ exists, then
\begin{equation*}
\Big\|\big(\mu^{1/k} +\nu^{1/k} + m^{1/k}\big)^k \Big\| \ll_k \xi M^{1-\frac{1}{k}}.
\end{equation*}
It follows from these and Lemma~\ref{lemma:erdosturan1} that
\begin{equation*}
\begin{split}
S_{21} \ll X^{2-\frac{2}{k}}V^{\varepsilon} \sum_M \frac{1}{M^{1+\frac{1}{k}}}\sum_{\substack{\nu,m \leq V/X\\ m\leq \nu \leq 2M}} \frac{1}{(\nu m)^{\frac{1}{2}+\frac{1}{2k}}} \Big( \xi M^{2-\frac{1}{k}}  + M^{\frac{2}{3}-\frac{1}{3k}} \big( \nu^{1/k}+m^{1/k}\big)^{1/3} \\
+ M^{\frac{1}{2}+\frac{1}{2k}} \big( \nu^{1/k}+m^{1/k}\big)^{-1/2} \Big)
\end{split}
\end{equation*}
(to handle the case $M=1/2$, we note that the conclusion of Lemma~\ref{lemma:erdosturan1} holds trivially for $W=1/2$). By an argument similar to our proof that \eqref{S121erdosturan} implies \eqref{S121bound}, we arrive at
\begin{equation}\label{S21bound}
S_{21}\ll \xi X^{1/k} V^{2-\frac{3}{k}+\varepsilon} + X^{4/3} V^{\frac{2}{3}-\frac{2}{k}+\varepsilon}.
\end{equation}
Next, to estimate the sum $S_{22}$ in \eqref{S2split}, we bound the integral in \eqref{S2def} via integration by parts and split the range of $|\Lambda_2|$ dyadically to deduce that, similarly to \eqref{S122dyadic}, we have
\begin{equation*}
S_{22} \ll X^{2-\frac{3}{k}} V^{\varepsilon} \sum_{L>\xi/2} \frac{1}{L} \sum_{\substack{\mu,\nu,m,n \leq V/X \\ m\leq \nu\leq \mu \\  L<|\Lambda_2|\leq 2L }} \frac{1}{(\mu\nu mn)^{\frac{1}{2}+\frac{1}{2k}} }.
\end{equation*}
If $m\leq \nu\leq \mu$ and $|\Lambda_2|\ll \mu^{1/k}$, then the definition \eqref{Lambda2def} of $\Lambda_2$ implies that
\begin{equation*}
n= \big( \mu^{1/k} +\nu^{1/k} + m^{1/k} \big)^k + O_k\big(|\Lambda_2|\mu^{1-\frac{1}{k}}\big).
\end{equation*}
Hence, as in our arguments below \eqref{mbinomial}, there exists a constant $\varepsilon_k>0$ such that if $m\leq \nu\leq \mu$ and $L \leq \varepsilon_k\mu^{1/k}$, then there are at most $1+O_k(L\mu^{1-\frac{1}{k}})$ integers $n$ satisfying $|\Lambda_2|\leq 2L$, and each such $n$ satisfies $n\asymp \mu$. The estimations leading up to \eqref{S122bound} then show that
\begin{equation*}
S_{22}\ll \xi^{-1} X^{1-\frac{1}{k}}V^{1-\frac{2}{k}+\varepsilon} |\log \xi| + V^{2-\frac{3}{k}+\varepsilon} (1+|\log \xi|).
\end{equation*}
From this, \eqref{S21bound}, and \eqref{S2split}, we arrive at
\begin{equation}\label{S2bound}
S_2 \ll \xi X^{1/k} V^{2-\frac{3}{k}+\varepsilon} + X^{4/3} V^{\frac{2}{3}-\frac{2}{k}+\varepsilon} + \xi^{-1} X^{1-\frac{1}{k}}V^{1-\frac{2}{k}+\varepsilon} |\log \xi| + V^{2-\frac{3}{k}+\varepsilon} (1+|\log \xi|).
\end{equation}
This finishes our estimation of $S_2$.

It is left to estimate $S_3$, which is defined by \eqref{S3def}. We bound the right-hand side of \eqref{S3def} by taking the absolute value of each term. By symmetry, we may then assume without loss of generality that $n\leq m\leq \nu\leq \mu$. Recalling the definition \eqref{X1def} of $X_1$, we estimate the integral in \eqref{S3def} via integration by parts and then use \eqref{a1def} to deduce that
\begin{equation*}
S_3 \ll X^{2-\frac{3}{k}} \sum_{\substack{\mu,\nu,m,n\leq V/X \\ n\leq m\leq \nu\leq \mu }} \frac{a_1}{\mu^{1/k}} \ll X^{2-\frac{3}{k}} V^{\varepsilon} \sum_{\substack{\mu,\nu,m,n\leq V/X \\ n\leq m\leq \nu\leq \mu }} \frac{1}{(\nu m n)^{\frac{1}{2}+\frac{1}{2k}} \mu^{\frac{1}{2}+\frac{3}{2k}}}.
\end{equation*}
We estimate the $n$-sum, $m$-sum, $\nu$-sum, and $\mu$-sum, in that order, to arrive at
\begin{equation}\label{S3bound}
S_3 \ll X^{2-\frac{3}{k}} V^{\varepsilon} \sum_{\substack{\mu,\nu,m \leq V/X \\  m\leq \nu\leq \mu }} \frac{1}{m^{1/k}\nu^{\frac{1}{2}+\frac{1}{2k}} \mu^{\frac{1}{2}+\frac{3}{2k}}} \ll X^{2-\frac{3}{k}} V^{\varepsilon} \sum_{\substack{\mu,\nu \leq V/X \\  \nu\leq \mu }} \frac{\nu^{\frac{1}{2}-\frac{3}{2k}}}{ \mu^{\frac{1}{2}+\frac{3}{2k}}} \ll V^{2-\frac{3}{k}+\varepsilon}.
\end{equation}

Now from \eqref{4thmmntsplit}, \eqref{Qsplit}, \eqref{S1bound}, \eqref{S2bound}, \eqref{S3bound}, we conclude that if $1\leq Y\leq T\leq X$, $V$ is defined by \eqref{Vdef}, and $0<\xi<1$ such that \eqref{xicondition} holds, then
\begin{equation}\label{4thmmntgeneral}
\begin{split}
\frac{1}{X}\int_X^{2X} |\Delta_k(x)|^4\,dx \ll
& X^{2-\frac{2}{k}} + X^{\frac{1}{2}-\frac{3}{2k}} V^{\frac{3}{2}-\frac{3}{2k}+\varepsilon} + \xi X^{1/k} V^{2-\frac{3}{k}+\varepsilon} + X^{4/3} V^{\frac{2}{3}-\frac{2}{k}+\varepsilon} \\
& + \xi^{-1} X^{1-\frac{1}{k}}V^{1-\frac{2}{k}+\varepsilon} |\log \xi| + V^{2-\frac{3}{k}+\varepsilon} (1+|\log \xi|)+ \frac{X^{2+\varepsilon}}{Y} \\
& + \frac{X^{4+\varepsilon}}{ Y^{2k+2}} + X^{\varepsilon}Y^{2k-4} + \frac{X^{2+\varepsilon}}{Y^4} + \frac{X^{4+\varepsilon}}{T^4} 
\end{split}
\end{equation}
under the assumption of the Lindel\"{o}f hypothesis. We now choose $\xi=X^{-\frac{1}{k}-\varepsilon}$, $T=X^{\frac{1}{2}+\frac{1}{2k}+\varepsilon}$, and $Y=X^{1/(k-1)}$, so that \eqref{Vdef} gives $V\ll X^{k/(k-1)}$, and the conditions $1\leq Y\leq T\leq X$ and \eqref{xicondition} are satisfied. With these choices for the parameters, \eqref{4thmmntgeneral} gives
\begin{equation*}
\frac{1}{X}\int_X^{2X} |\Delta_k(x)|^4\,dx \ll X^{2-\frac{1}{k-1}+\varepsilon}.
\end{equation*}
This completes the proof of Theorem~\ref{thm:deltakfourthmmnt}.

\section{Intervals containing no sign changes}\label{section:main}

To complete the proofs of Theorems~\ref{thm:deltaknosignchangeRH}, \ref{thm:deltaknosignchange}, and \ref{thm:delta3nosignchangeRH}, we first bound the integral
\begin{equation}\label{supmeanvalue}
\frac{1}{X}\int_X^{2X} \sup_{0\leq h\leq H} \Big(\Delta_k(x+h)-\Delta_k(x)\Big)^2 \,dx.
\end{equation}
We do this by applying a method of Heath-Brown and Tsang~\cite{hbtsang} that enables us to use Theorems~\ref{thm:dkinshortintervalsnosup} and \ref{thm:dkinshortintervalsRH} to bound \eqref{supmeanvalue}.

Suppose that $1\leq H\leq X/8$. We write $H$ as
\begin{equation}\label{H2lb}
H=2^{\ell}b
\end{equation}
for some unique $\ell,b$ such that $\ell$ is a nonnegative integer and $1\leq b<2$. The definition \eqref{deltakdef} of $\Delta_k(x)$ implies that
\begin{equation}\label{deltakdefpoly}
\Delta_k(x) = \sum_{n\leq x} d_k(n) - xP_k(\log x)
\end{equation}
for some polynomial $P_k$ of degree $k-1$. Thus $\Delta_k(x)$ is continuous except at points $x=n$ with $n$ an integer, where it is continuous from the right and has left-hand limit $\Delta_k(n)-d_k(n)$. It follows that there is an $h_0\in [0,H]$ such that either
\begin{equation}\label{EVT1}
\sup_{0\leq h\leq H} |\Delta_k(x+h)-\Delta_k(x)|^2 = |\Delta_k(x+h_0)-\Delta_k(x)|^2
\end{equation}
or
\begin{equation}\label{EVT2}
\sup_{0\leq h\leq H} |\Delta_k(x+h)-\Delta_k(x)|^2 = |\Delta_k(x+h_0)-d_k(x+h_0)-\Delta_k(x)|^2.
\end{equation}

Suppose first that \eqref{EVT1} holds. By \eqref{H2lb} and the fact that $0\leq h_0\leq H$, we have
\begin{equation}\label{h0}
jb\leq h_0 \leq (j+1)b
\end{equation}
for some integer $j$ satisfying $0\leq j\leq 2^{\ell}-1$. The expression \eqref{deltakdefpoly} and the mean value theorem of differential calculus imply that
\begin{equation*}
\Delta_k(u_2)-\Delta_k(u_1) = \sum_{u_1<n\leq u_2} d_k(n) + O((u_2-u_1) \log^k(X+2))
\end{equation*}
for $1\leq u_1\leq u_2 \ll X$. Since $d_k(n)\geq 0$ for all $n$, it follows that
\begin{equation}\label{deltakdifferencebound}
\Delta_k(u_2) \geq \Delta_k(u_1)- O((u_2-u_1) \log^k(X+2))
\end{equation}
for $1\leq u_1\leq u_2 \ll X$. If $\Delta_k(x+h_0)\geq \Delta_k(x)$, then \eqref{h0} and \eqref{deltakdifferencebound} give
\begin{equation*}
0\leq \Delta_k(x+h_0)-\Delta_k(x) \leq \Delta_k(x+(j+1)b)-\Delta_k(x) + O(b \log^k(X+2)),
\end{equation*}
while if $\Delta_k(x+h_0)\leq \Delta_k(x)$, then \eqref{h0} and \eqref{deltakdifferencebound} imply
\begin{equation*}
0\geq \Delta_k(x+h_0)-\Delta_k(x) \geq \Delta_k(x+ jb)-\Delta_k(x) - O(b \log^k(X+2)).
\end{equation*}
In either case, we have
\begin{equation*}
|\Delta_k(x+h_0)-\Delta_k(x)| \leq \max_{0\leq j\leq 2^{\ell}}|\Delta_k(x+ jb)-\Delta_k(x)| + O( \log^k(X+2)).
\end{equation*}
From this and \eqref{EVT1}, we arrive at
\begin{equation}\label{supasmax}
\sup_{0\leq h\leq H} |\Delta_k(x+h)-\Delta_k(x)|^2 \ll \max_{0\leq j\leq 2^{\ell}}|\Delta_k(x+ jb)-\Delta_k(x)|^2 +  O(X^{\varepsilon}).
\end{equation}

We have shown that if \eqref{EVT1} holds, then \eqref{supasmax} is true. Now suppose that \eqref{EVT2} holds and $x+h_0$ is a positive integer. Then
$$
\Delta_k(x+h_0)-d_k(x+h_0)-\Delta_k(x)<0
$$
since otherwise $|\Delta_k(x+h_0)-\Delta_k(x)|>|\Delta_k(x+h_0)-d_k(x+h_0)-\Delta_k(x)|$, which contradicts \eqref{EVT2}. Hence \eqref{h0} and \eqref{deltakdifferencebound} imply
$$
0>\Delta_k(x+h_0)-d_k(x+h_0)-\Delta_k(x) \geq \Delta_k(x+ jb)-d_k(x+h_0)-\Delta_k(x) - O(b \log^k(X+2)),
$$
and \eqref{supasmax} again follows because $d_k(x+h_0)\ll X^{\varepsilon}$. We have thus proved that \eqref{supasmax} holds in either case. Consequently, for each $x$ with $X\leq x\leq 2X$, there is an integer $j_0=j_0(x)$ such that
\begin{equation}\label{j0bound}
0\leq j_0\leq 2^{\ell}
\end{equation}
and
\begin{equation}\label{supasmax2}
\sup_{0\leq h\leq H} |\Delta_k(x+h)-\Delta_k(x)|^2 \ll  |\Delta_k(x+ j_0b)-\Delta_k(x)|^2 +  O(X^{\varepsilon}).
\end{equation}

This, by itself, does not enable us to use Theorems~\ref{thm:dkinshortintervalsnosup} or \ref{thm:dkinshortintervalsRH} to bound \eqref{supmeanvalue} because $j_0$ might depend on $x$. To get around this difficulty, we use the technique of Heath-Brown and Tsang~\cite{hbtsang} that uses the binary expansion of $j_0$ and the Cauchy-Schwarz inequality to bound the right-hand side of \eqref{supasmax2} by a sum of quantities of the form $|\Delta_k(x+ h_1)-\Delta_k(x+h_2)|^2$ with $h_1$ and $h_2$ independent of $x$.

Since $j_0$ is an integer satisfying \eqref{j0bound}, it has a unique binary expansion
\begin{equation}\label{jdefforlemma:dyadic}
j_0 = \sum_{\mu \in U} 2^{\ell-\mu}
\end{equation}
for some subset $U$ of $\{0,1,2,\dots,\ell\}$. We let
\begin{equation}\label{numudefforlemma:dyadic}
\nu_{\mu} = \sum_{\substack{m\in U \\ m<  \mu}} 2^{\mu-m}
\end{equation}
for each $\mu\in U$, and write $\Delta_k(x+j_0b)-\Delta_k(x)$ as a telescoping sum
\begin{equation*}
\Delta_k(x+j_0b)-\Delta_k(x) = \sum_{\mu\in U} \Big( \Delta_k \big(x+(\nu_{\mu}+1) 2^{\ell-\mu}  b \big)-\Delta_k\big(x+  \nu_{\mu} 2^{\ell-\mu} b \big) \Big).
\end{equation*}
It follows from this and the Cauchy-Schwarz inequality that
\begin{equation}\label{cauchyforlemma:dyadic}
|\Delta_k(x+j_0b)-\Delta_k(x)|^2 \leq (\ell+1) \sum_{\mu\in U} \big| \Delta_k \big(x+(\nu_{\mu}+1) 2^{\ell-\mu}  b \big)-\Delta_k\big(x+  \nu_{\mu} 2^{\ell-\mu} b \big) \big|^2.
\end{equation}
Note that if $0\in U$, then $U=\{0\}$ by \eqref{j0bound} and \eqref{jdefforlemma:dyadic}. In this case, $\nu_0=0$ by \eqref{numudefforlemma:dyadic}. On the other hand, if $0\not\in U$, then $\nu_{\mu}<2^{\mu}$ by \eqref{numudefforlemma:dyadic}. In either case, it holds that $0\leq \nu_{\mu}<2^{\mu}$ for all $\mu\in U$. Thus, by including all possible values for $\mu$ and $\nu_{\mu}$, we deduce from \eqref{cauchyforlemma:dyadic} that
\begin{equation*}
|\Delta_k(x+j_0b)-\Delta_k(x)|^2 \leq (\ell+1) \sum_{0\leq \mu\leq \ell} \sum_{0\leq \nu<2^{\mu}} |\Delta_k(x+(\nu+1)2^{\ell-\mu}b)-\Delta_k(x+\nu 2^{\ell-\mu}b)|^2,
\end{equation*}
where the indices of summation $\mu$ and $\nu$ run through integers. From this and \eqref{supasmax2}, we arrive at
\begin{align*}
& \frac{1}{X}\int_X^{2X} \sup_{0\leq h\leq H} |\Delta_k(x+h)-\Delta_k(x)|^2 \,dx\\
& \ll (\ell+1) \sum_{0\leq \mu\leq \ell} \sum_{0\leq \nu<2^{\mu}} \frac{1}{X}\int_X^{2X}|\Delta_k(x+(\nu+1)2^{\ell-\mu}b)-\Delta_k(x+\nu 2^{\ell-\mu}b)|^2\,dx +  O(X^{\varepsilon}).
\end{align*}
This and a change of variables $x\mapsto x-\nu 2^{\ell-\mu}b$ leads to
\begin{align*}
& \frac{1}{X}\int_X^{2X} \sup_{0\leq h\leq H} |\Delta_k(x+h)-\Delta_k(x)|^2 \,dx\\
& \ll (\ell+1) \sum_{0\leq \mu\leq \ell} \sum_{0\leq \nu<2^{\mu}} \frac{1}{X}\int_{X+\nu 2^{\ell-\mu}b}^{2X+\nu 2^{\ell-\mu}b}|\Delta_k(x+ 2^{\ell-\mu}b)-\Delta_k(x )|^2\,dx +  O(X^{\varepsilon}).
\end{align*}

To bound the latter integral, we may apply any of Corollary~\ref{cor:d3inshortintervals}, Corollary~\ref{cor:dkinshortintervalsLH}, or Theorem~\ref{thm:dkinshortintervalsRH} because $1\leq 2^{\ell-\mu}b \leq X/8$ for all $\mu\in\{0,1,\dots, \ell\}$ by \eqref{H2lb} and our assumption that $H\leq X/8$. Applying Corollary~\ref{cor:d3inshortintervals} gives
\begin{align*}
\frac{1}{X}\int_X^{2X} \sup_{0\leq h\leq H} |\Delta_3(x+h)-\Delta_3(x)|^2 \,dx & \ll_{\varepsilon} (\ell+1) \sum_{0\leq \mu\leq \ell} \sum_{0\leq \nu<2^{\mu}}2^{\ell-\mu}b X^{\frac{1}{6}+\varepsilon}   +   X^{\varepsilon}  \\
& = (\ell+1)^2 2^{\ell }b X^{\frac{1}{6}+\varepsilon} +   X^{\varepsilon} .
\end{align*}
From this and \eqref{H2lb}, we deduce that if $1\leq H\leq X/8$, then
\begin{equation*}
\frac{1}{X}\int_X^{2X} \sup_{0\leq h\leq H} |\Delta_3(x+h)-\Delta_3(x)|^2 \,dx \ll_{\varepsilon} H X^{\frac{1}{6}+\varepsilon}.
\end{equation*}
This bound holds true unconditionally, i.e., independently of any unproved conjecture. Similarly, applying Corollary~\ref{cor:dkinshortintervalsLH} instead of Corollary~\ref{cor:d3inshortintervals}, we see that if $1\leq H\leq X/8$, then
\begin{equation}\label{eqn:dkshortintervalssup}
\frac{1}{X}\int_X^{2X} \sup_{0\leq h\leq H} |\Delta_k(x+h)-\Delta_k(x)|^2 \,dx \ll_{k,\varepsilon} H X^{ \varepsilon}
\end{equation}
provided that the Lindel\"{o}f hypothesis is true. On the other hand, applying Theorem~\ref{thm:dkinshortintervalsRH} and arguing in a similar way, we deduce that if $1\leq H\leq X/8$, then
\begin{equation}\label{eqn:dkshortintervalssupRH}
\frac{1}{X}\int_X^{2X} \sup_{0\leq h\leq H} |\Delta_k(x+h)-\Delta_k(x)|^2 \,dx \ll_{k } H (\log X)^{k^2+2}    +   X^{\varepsilon}
\end{equation}
provided that the Riemann hypothesis is true.

We now have all the ingredients needed to prove Theorems~\ref{thm:deltaknosignchangeRH}, \ref{thm:deltaknosignchange}, and \ref{thm:delta3nosignchangeRH} using the method of Heath-Brown and Tsang~\cite{hbtsang} for finding intervals containing no sign changes. Let $\eta>0$ be an arbitrarily small (fixed) constant. Define $G_k(x)$ by
\begin{equation}\label{Gkdef}
G_k(x):= |\Delta_k(x)|-\left(\frac{1}{2}C_k-\eta \right)x^{\frac{1}{2}-\frac{1}{2k}},
\end{equation}
where the constant $C_k$ is defined by \eqref{Ckdef}. Let $H\geq 1$ be a parameter to be chosen later, and define $W_k(x)$ by 
\begin{equation}\label{Wkdef}
W_k(x)=W_k(x;H):= G_k^2(x)-\sup_{0\leq h\leq H} \Big(G_k(x+h)-G_k(x)\Big)^2 -\bigg( \frac{1}{2}C_k x^{\frac{1}{2}-\frac{1}{2k}}\bigg)^2.
\end{equation}
Let $\mathcal{S}$ be the set
\begin{equation}\label{Sdef}
\mathcal{S}:=\{ x\in [X,2X]: W_k(x)>0\}.
\end{equation}
By the definition \eqref{Wkdef} of $W_k$, if $x\in \mathcal{S}$, then
\begin{enumerate}
\item[(i)] $|G_k(x)|> \underset{0\leq h\leq H}{\sup} |G_k(x+h)-G_k(x)|$, and
\item[(ii)] $|G_k(x)|>\frac{1}{2}C_k x^{\frac{1}{2}-\frac{1}{2k}}$.
\end{enumerate}
Property (i) implies that $G_k(x)$ has the same sign as $G_k(y)$ for all $y\in [x,x+H]$. Property (ii) implies that $G_k(x)>0$, since otherwise the definition \eqref{Gkdef} of $G_k$ would imply
\begin{equation*}
|G_k(x)|= \left(\frac{1}{2}C_k-\eta \right)x^{\frac{1}{2}-\frac{1}{2k}} - |\Delta_k(x)|<\frac{1}{2}C_k x^{\frac{1}{2}-\frac{1}{2k}},
\end{equation*}
which negates (ii). Thus, if $x\in \mathcal{S}$, then $G_k(y)>0$ for all $y\in [x,x+H]$. By \eqref{Gkdef}, this means that if $x\in \mathcal{S}$, then
\begin{equation}\label{Deltaklower}
|\Delta_k(y)| > \left(\frac{1}{2}C_k-\eta \right)y^{\frac{1}{2}-\frac{1}{2k}}
\end{equation}
for all $y\in [x,x+H]$. If \eqref{Deltaklower} holds for all $y\in [x,x+H]$, then $\Delta_k$ does not change sign in $[x,x+H]$ because if $\Delta_k$ has a jump discontinuity at $y$, then the jump has size $d_k(y) \ll y^{\varepsilon}$. Hence, to show the existence of an interval of length $H$ on which $\Delta_k$ does not change sign, it suffices to prove that $\mathcal{S}$ is nonempty. We will in fact do more than this by finding a lower bound for the Lebesgue measure of $\mathcal{S}$. We will choose $H=X^{1-\frac{1}{k}-\varepsilon}$ to prove Theorem~\ref{thm:deltaknosignchange} and $H=c_0\eta X^{1-1/k}(\log X)^{-k^2-2}$ for a suitable constant $c_0>0$ to prove Theorems~\ref{thm:deltaknosignchangeRH} and \ref{thm:delta3nosignchangeRH}.

To find a lower bound for the Lebesgue measure of $\mathcal{S}$, first observe that the definitions \eqref{Wkdef} of $W_k$ and \eqref{Sdef} of $\mathcal{S}$ and the Cauchy-Schwarz inequality imply 
\begin{equation}\label{Wkupper}
\int_X^{2X}W_k(x)\,dx \leq \int_{\mathcal{S}} W_k(x)\,dx \leq \int_{\mathcal{S}} G^2_k(x)\,dx \leq \mathcal{M}^{1/2} \Bigg( \int_X^{2X} G_k^4(x) \,dx\Bigg)^{1/2},
\end{equation}
where $\mathcal{M}$ is the Lebesgue measure of $\mathcal{S}.$ 
Therefore, a lower bound for the integral of $W_k $ together with an upperbound for the fourth moment of $G_k$ gives a lowerbound for $\mathcal{M}$. Now the definition \eqref{Gkdef} of $G_k$, the inequality $|a+b|^4\ll |a|^4+|b|^4$, and Theorem~\ref{thm:deltakfourthmmnt} give
\begin{equation}\label{Gk4thmmnt}
\int_X^{2X} G_k^4(x)\,dx \ll   X^{3-\frac{1}{k-1}+\varepsilon}
\end{equation}
provided that the Lindel\"{o}f hypothesis is true.

It is left to find a lower bound for the integral of $W_k$ in \eqref{Wkupper}. We do this by estimating the integrals of each of the terms in the definition \eqref{Wkdef} of $W_k$. For the first term, Tong's formula~\eqref{Tong2ndmoment}, the definition \eqref{Gkdef} of $G_k$, and the Cauchy-Schwarz inequality imply
\begin{equation}\label{Tong2ndmoment2}
\begin{split}
\int_X^{2X} \Big( G_k(x) \Big)^2 \,dx \geq
& \int_X^{2X} | \Delta_k(x) |^2 \,dx + \left(\frac{1}{2}C_k-\eta \right)^2 \int_X^{2X} x^{1-\frac{1}{k}}\,dx \\
& - 2 \Bigg(\int_X^{2X} | \Delta_k(x) |^2 \,dx\Bigg)^{1/2} \Bigg( \int_X^{2X} \left(\frac{1}{2}C_k-\eta \right)^2 x^{1-\frac{1}{k}}\,dx\Bigg)^{1/2} \\
\geq & (1+o(1))\left(\frac{1}{2}C_k+\eta \right)^2 \int_X^{2X} x^{1-\frac{1}{k}}\,dx.
\end{split}
\end{equation}
To estimate the integral of the second term in \eqref{Wkdef}, observe that the mean value theorem of differential calculus implies
\begin{equation*}
(x+h)^{\frac{1}{2}-\frac{1}{2k}}-x^{\frac{1}{2}-\frac{1}{2k}} \ll_k h x^{-\frac{1}{2}-\frac{1}{2k}}
\end{equation*}
for $h\geq 0$. It follows from this, the definition \eqref{Gkdef} of $G_k$, and the inequalities $||a|-|b||\leq |a-b|$ and $|a+b|^2\ll |a|^2+|b|^2$ that
\begin{equation}\label{eqn:Gkshortintervalbound}
\sup_{0\leq h\leq H} \Big(G_k(x+h)-G_k(x)\Big)^2 \ll \sup_{0\leq h\leq H} \Big(\Delta_k(x+h)-\Delta_k(x)\Big)^2 + H^2x^{-1-\frac{1}{k}}.
\end{equation}
We will use this shortly to show that we can choose the parameter $H\geq 1$ in such a way that
\begin{equation}\label{eqn:chooseH}
\int_X^{2X}\sup_{0\leq h\leq H} \Big(G_k(x+h)-G_k(x)\Big)^2 \leq \frac{1}{2}C_k\eta \int_X^{2X} x^{1-\frac{1}{k}}\,dx.
\end{equation}
If \eqref{eqn:chooseH} holds, then \eqref{Wkdef}, \eqref{Tong2ndmoment2}, and \eqref{eqn:chooseH} imply
\begin{equation}\label{Wklower}
\int_X^{2X}W_k(x)\,dx \geq  (1+o_{\eta}(1))\eta \left(\frac{1}{2}C_k+\eta\right) \int_X^{2X} x^{1-\frac{1}{k}}\,dx.
\end{equation}

From this, \eqref{Wkupper}, and \eqref{Gk4thmmnt}, we deduce that if $H\geq 1$ satisfies \eqref{eqn:chooseH} and LH is true, then
\begin{equation*}
\mathcal{M} \gg  X^{1+\frac{1}{k-1}-\frac{2}{k}-\varepsilon},
\end{equation*}
where we recall that $\mathcal{M}$ is the Lebesgue measure of $\mathcal{S}$. Since each $x\in \mathcal{S}$ has the property that \eqref{Deltaklower} holds for all $y\in [x,x+H]$, it follows that there are at least $\gg \mathcal{M}/H$ disjoint subintervals of $[X,2X]$ of length $H$ such that \eqref{Deltaklower} holds for all $y$ in the subinterval. If $k\geq 3$ and the Lindel\"{o}f hypothesis is true, then \eqref{eqn:dkshortintervalssup} and \eqref{eqn:Gkshortintervalbound} imply that $H=X^{1-\frac{1}{k}-\varepsilon}$ satisfies \eqref{eqn:chooseH} for large enough $X$, and this proves Theorem~\ref{thm:deltaknosignchange}. Moreover, if $k\geq 3$ and the Riemann hypothesis is true, then, by \eqref{eqn:dkshortintervalssupRH} and \eqref{eqn:Gkshortintervalbound}, there exists a small enough constant $c_0>0$ depending only on $k$ such that if
\begin{equation}\label{Hchoice}
H=c_0\eta X^{1-\frac{1}{k}}(\log X)^{-k^2-2},
\end{equation}
then \eqref{eqn:chooseH} holds for large enough $X$. This completes the proof of Theorem~\ref{thm:deltaknosignchangeRH}.

To prove Theorem~\ref{thm:delta3nosignchangeRH}, we argue as in equation (7.5) of \cite{CTZ} and use H\"{o}lder's inequality instead of the Cauchy-Schwarz inequality in \eqref{Wkupper} to deduce that
\begin{equation}\label{Wkupper2}
\int_X^{2X}W_k(x)\,dx \leq \int_{\mathcal{S}} W_k(x)\,dx \leq \int_{\mathcal{S}} G^2_k(x)\,dx \leq \mathcal{M}^{1/3} \Bigg( \int_X^{2X} |G_k(x)|^3 \,dx\Bigg)^{2/3}.
\end{equation}
The definition \eqref{Gkdef} of $G_k$ with $k=3$, the inequality $|a+b|^3\ll |a|^3+|b|^3$, and \eqref{HBbound} give
\begin{equation*}
\int_X^{2X} |G_3(x)|^3 \,dx \ll X^{2+\varepsilon}.
\end{equation*}
From this, \eqref{Wklower}, and \eqref{Wkupper2}, we deduce that
\begin{equation*}
\mathcal{M} \gg X^{1-\varepsilon} 
\end{equation*}
for $k=3$ provided that the Riemann hypothesis is true and $H$ is given by \eqref{Hchoice} with $k=3$. It follows that there are at least $\gg \mathcal{M}/H \gg X^{\frac{1}{3}-\varepsilon}$ disjoint subintervals of $[X,2X]$ of length $H$ such that \eqref{Deltaklower} holds for all $y$ in the subinterval. This proves Theorem~\ref{thm:delta3nosignchangeRH}.

\end{document}